\documentclass[reqno]{amsart}

\usepackage[utf8]{inputenc}
\usepackage{amssymb, amsmath, amsthm, color, enumerate, mathabx, mathrsfs}
\usepackage{bbm}
\usepackage{pdfpages}
\usepackage{tcolorbox}
\usepackage{esint} 
\usepackage{mathtools}
\usepackage{subcaption}
\usepackage{marginnote}
\usepackage{chngcntr}

\usepackage{multicol}
\usepackage{wrapfig}

\counterwithin{figure}{section}

\numberwithin{equation}{section}
\theoremstyle{plain}
\newtheorem{theorem}{Theorem}[section]
\newtheorem{lemma}[theorem]{Lemma}
\newtheorem{proposition}[theorem]{Proposition}
\newtheorem{corollary}[theorem]{Corollary}

\newtheorem*{claim}{Claim}
\theoremstyle{definition}
\newtheorem{definition}[theorem]{Definition}
\newtheorem{example}[theorem]{Example}
\newtheorem{remark}[theorem]{Remark}

\newcommand{\bbA}{\mathbb{A}}
\newcommand{\bbB}{\mathbb{B}}
\newcommand{\C}{\mathbb{C}}

\newcommand{\N}{\mathbb{N}}

\newcommand{\R}{\mathbb{R}}

\newcommand{\Z}{\mathbb{Z}}

\newcommand{\cA}{\mathcal{A}}
\newcommand{\cB}{\mathcal{B}}
\newcommand{\cC}{\mathcal{C}}

\newcommand{\cM}{\mathcal{M}}

\newcommand{\cQ}{\mathcal{Q}}

\newcommand{\bdnu}{\boldsymbol{\nu}}

\newcommand{\ud}{\mathrm{d}}
\def\inn#1#2{\langle#1,#2\rangle}
\newcommand{\floor}[1]{\lfloor #1 \rfloor }
\newcommand{\ceil}[1]{\lceil #1 \rceil }
\newcommand{\supp}{\mathrm{supp}\,}

\newcommand{\dist}{\mathrm{dist}}

\newcommand{\sgn}{\mathrm{sgn}}

\newcommand{\norm}[1]{\|#1\|}
\newcommand{\Norm}[1]{\bigg\|#1\bigg\|}
\newcommand{\normLpRd}[3]{\norm{#1}_{L^{#2}(\R^{#3})}}
\newcommand{\NormLpRd}[3]{\Norm{#1}_{L^{#2}(\R^{#3})}}

\newcommand{\Abs}[1]{\biggl| #1 \biggr|}

\newcommand{\Brac}[1]{\biggl( #1 \biggr)}

\newcommand{\Chi}{\Tilde{\chi}}
\newcommand{\tang}[2]{\Delta(#1, #2)}
\newcommand{\ft}{\; \widehat{} \; }

\let\angle\measuredangle

\counterwithin{figure}{section} 

\begin{document}

\title[The circular maximal theorem]{Bourgain's proof of the circular maximal theorem: an exposition}

\author[H. Y. G\"und\"uz]{ Hal\.{i}t Y\.{i}\u{g}\.{i}t G\"und\"uz }
\address{Hal\.{\.i}t Y\.{\.i}\u{g}\.{\.i}t G\"und\"uz: School of Mathematics, The Watson Building, University of Birmingham, Edgbaston, Birmingham, B15 2TT, England.}
\email{hyg637@student.bham.ac.uk}

\author[J. Hickman]{ Jonathan Hickman }
\address{Jonathan Hickman: School of Mathematics and Maxwell Institute for Mathematical Sciences, James Clerk Maxwell Building, The King's Buildings, Peter Guthrie Tait Road, Edinburgh, EH9 3FD, UK.}
\email{jonathan.hickman@ed.ac.uk}

\begin{abstract} We present an exposition of Bourgain's celebrated 1986 proof of the circular maximal theorem. 
    
\end{abstract}

\maketitle




\section{Introduction} 




\subsection{The spherical and circular maximal theorems} For $d \geq 2$, let $\sigma$ denote the surface measure on the unit sphere $S^{d-1}$ in $\R^d$, normalised to have mass $1$. Given $f \in C_c(\R^d)$, define the \textit{spherical means}
\begin{equation}\label{eq: spherical mean}
    A f(x, r) \coloneq  f*\sigma_r(x) = \int_{S^{d-1}} f(x - r \omega)\,\ud \sigma(\omega), \qquad (x,r) \in \R^d \times (0, \infty),
\end{equation}
so that $Af(x,r)$ corresponds to the average of $f$ over the sphere 
\begin{equation*}
    C(x, r) \coloneq \{y \in \R^d : |x - y| = r \}. 
\end{equation*}
Here $\sigma_r$ denotes the dilated measure, defined by the action $\inn{\sigma_r}{f} \coloneq \inn{\sigma}{f(r \; \cdot \,)}$ for $f \in C_c(\R^d)$ and $r > 0$. Finally, we define the associated \textit{spherical} (or, in the case $d = 2$, the \textit{circular}) \textit{maximal function}
\begin{equation}\label{eq: spherical max}
    M f(x) \coloneq \sup_{r > 0} |A f(x, r)|. 
\end{equation}

The operator $M$ is a singular variant of the classical Hardy--Littlewood maximal function. Accordingly, the following foundational theorem of Stein~\cite{Stein1976} establishes a singular variant of the Hardy--Littlewood maximal theorem. 

\begin{theorem}[Spherical maximal theorem: Stein, 1976 \cite{Stein1976}]\label{thm: Stein} For all $d \geq 3$ and $p > \frac{d}{d-1}$, we have\footnote{Given a list of objects $L$ and non-negative real numbers $A$, $B$, we write $A \lesssim_L B$, $B \gtrsim_L A$ or $A = O_L(B)$ if $A \leq C_L B$ where $C_L > 0$ is a constant depending only on the objects listed in $L$ and possibly a choice of dimension $d$.}
\begin{equation}\label{eq: Stein thm}
    \|M f\|_{L^p(\R^d)} \lesssim_p \|f\|_{L^p(\R^d)} \qquad \textrm{for all $f \in C_c(\R^d)$.} 
\end{equation}
\end{theorem}

Stein's elegant argument \cite{Stein1976} relies heavily on Fourier analysis. This is somewhat surprising, given that the operator $M$ is positive and there is no overtly oscillatory feature of the problem. Another interesting aspect of Theorem~\ref{thm: Stein} is that the arguments used in \cite{Stein1976} fail to establish any nontrivial $L^p$ bound in the $d = 2$ case. Indeed, determining the $L^p$ boundedness of the \textit{circular} maximal function is a much more difficult problem, which was eventually settled in a celebrated work of Bourgain~\cite{Bourgain1986} some ten years later.

\begin{theorem}[Circular maximal theorem: Bourgain, 1986 \cite{Bourgain1986}]\label{thm: Bourgain} For $p > 2$, we~have
\begin{equation*}
    \|M f\|_{L^p(\R^2)} \lesssim_p \|f\|_{L^p(\R^2)} \qquad \textrm{for all $f \in C_c(\R^2)$.}  
\end{equation*}
\end{theorem}

We remark that, by testing the inequality \eqref{eq: Stein thm} against simple examples (see \S\ref{sec: nec} below), it is straightforward to show that $M$ fails to be bounded on $L^p(\R^d)$ whenever $d \geq 2$ and $p \leq \frac{d}{d-1}$. Thus, the range of $p$ in both Theorem~\ref{thm: Stein} and Theorem~\ref{thm: Bourgain} is sharp. 

 Bourgain's \cite{Bourgain1986} proof of Theorem~\ref{thm: Bourgain} is a beautiful and ingenious synthesis of both Fourier analysis and geometry. Loosely speaking, it centres around analysing an underlying set $\cC$ of circles $C(x,r_x)$, where the radii $r_x$ are chosen to realise the supremum in \eqref{eq: spherical max}. If few pairs of circles in $\cC$ are close to being tangent, then this information can be fed into Stein's Fourier analytic argument \cite{Stein1976} to obtain improved estimates and thereby establish Theorem~\ref{thm: Bourgain} in this case. The crux, therefore, is to deal with the remaining situation where many pairs of circles in $\cC$ are close to being tangent. Here classical Fourier analysis is less effective and Bourgain~\cite{Bourgain1986} instead relies on intricate geometric/combinatorial arguments.




\subsection{Objectives} The goal of this article is to present an exposition of Bourgain's original proof \cite{Bourgain1986} of Theorem~\ref{thm: Bourgain}. For this, we focus on the following three themes.

\begin{enumerate}[1)]
    \item \textit{Fourier analysis vs geometry}. A striking feature of Bourgain's proof~\cite{Bourgain1986} is the manner in which it combines both Fourier analysis and geometry. Here we seek to explain the relative strengths and weaknesses of Fourier analytic and geometric tools in this context. A preliminary goal is to simply understand why Fourier analysis is a natural tool for this problem in the first place. 
    \item \textit{Circles vs spheres}. We highlight various phenomena encountered specifically in the $d = 2$ case, which explain why the circular maximal theorem is more difficult than the spherical. We also explore the relationship between the arguments used by Bourgain~\cite{Bourgain1986} in the $d = 2$ case with those used by Stein~\cite{Stein1976} for $d \geq 3$. We recast the key steps of the latter in the formalism of the former and, in this way, see how Bourgain's arguments~\cite{Bourgain1986} are a natural development of Stein's~\cite{Stein1976}. 
    \item \textit{Streamlining and clarifying}. Although we present all the key ideas, we do not follow Bourgain's arguments~\cite{Bourgain1986} to the letter. We often streamline arguments, adopt more modern formalisms, and reinterpret or rework methods for expository purposes.
\end{enumerate}

Since the publication of Bourgain's work~\cite{Bourgain1986}, there has been sustained and significant interest in geometric maximal operators associated to curves and surfaces in $\R^d$. Indeed, the broad topic (which naturally encompasses the study of the Kakeya and Nikodym conjectures: see, for example, \cite{Zahl2026}) is central to modern harmonic analysis and geometric measure theory.  Much work in this vein has relied on either predominantly Fourier analytic \cite{MSS1992, KLO2023} or predominantly geometric/combinatorial \cite{KW1999, Wolff1997} methods. However, recent papers have highlighted the limitations of purely Fourier analytical approaches (often in the guise of \textit{local smoothing}: see \cite{MSS1992, BHS2021}), and have shown that synthesising Fourier analysis with geometry/combinatorics can lead to deep and interesting new results. Two concrete examples are:
\begin{itemize}
    \item \textit{Space curve maximal functions}. Known Fourier analytic techniques establish partial results, but do not lead to sharp maximal estimates in high dimensions \cite{GMO, BH}.
    \item \textit{Multiparameter maximal functions}. Sharp bounds have only been obtained by combining both Fourier analytic and geometric/combinatorial approaches \cite{LLO2025, LLO2025a, CGY, PYZ, HZ, Zahl2026}.
\end{itemize}
This trend suggests reassessing the relationship between Fourier analytic and geometric/combinatorial arguments in the study of geometric maximal functions, and examining the extent to which these arguments can be synthesised. It is therefore natural and timely to return to Bourgain's original argument~\cite{Bourgain1986}: the proof provides a striking illustration of such a synthesis.



\subsection{Further reading} The study of geometric maximal functions has a rich history in harmonic analysis, spanning over 50 years. Since this is primarily an expository article, we do not attempt to provide a comprehensive survey of the literature. There are, however,  a number of existing survey articles and expository papers on geometric maximal functions and related topics: see, for instance, \cite{SW1978, Wolff1999, BHS2021, RS}. The spherical maximal theorem is also treated in a number of textbooks: see \cite[Chapter XI]{SteinBook}, \cite[\S2.4, \S8.3]{SoggeBook}, \cite[\S6.5]{Grafakos_book}. 

Alternative proofs of Theorem~\ref{thm: Stein} and Theorem~\ref{thm: Bourgain} and closely related results can be found, for instance, in \cite{Marstrand1987, MSS1992, Schlag1998, HJ2025}. Arguably the most notable of these is the landmark paper of Mockenhaupt--Seeger--Sogge \cite{MSS1992}, which further develops the \textit{local smoothing} perspective on the circular maximal function introduced by Sogge~\cite{Sogge1991}. The framework of \cite{MSS1992} has been highly influential, leading to the introduction of decoupling theory \cite{Wolff2000, BD2015} and important advances in the theory of Littlewood--Paley-type inequalities \cite{GWZ2020}. It has also formed the basis for many recent advances in the study of geometric maximal operators: see, for instance, \cite{KLO2022, KLO2023, BHGS2025, LLO2025, LLO2025a, CGY}. We refer the reader to \cite{SoggeBook, BHS2021} for further information.

We highlight \cite{Sogge1991, LP2011} as examples of works directly inspired by Bourgain's proof of the circular maximal theorem~\cite{Bourgain1986}. It is also interesting to note the similarities between the arguments used by Bourgain in \cite{Bourgain1986} and those used later in his groundbreaking 1991 study of Fourier restriction and Kakeya/Nikodym phenomena \cite{Bourgain1991, Bourgain1991a}. 




\subsection*{Structure of the article}  In \S\ref{sec: nec}, we briefly discuss necessary conditions showing the sharpness of Theorem~\ref{thm: Stein} and Theorem~\ref{thm: Bourgain}. In \S\ref{sec: Fourier prelim}, we introduce basic functional and Fourier analytical tools used to study the spherical and circular maximal functions. In \S\ref{sec: Stein}, we give a proof of the key $L^2$ bounds used in the proof of Stein's spherical maximal theorem. For this, we use a formalism which we later exploit in the proof of Bourgain's theorem. We also attempt to motivate the use of Fourier analysis in the study of geometric maximal functions, in part by highlighting limitations of purely geometrical arguments. In \S\ref{sec:geometriclemmas}, we introduce additional geometric prerequisites needed for the proof of Theorem~\ref{thm: Bourgain}. In \S\ref{sec: main argument}, we combine reductions and Fourier analytic tools from \S\ref{sec: Fourier prelim} and \S\ref{sec: Stein} with the geometric results from \S\ref{sec:geometriclemmas} to conclude the proof of Theorem~\ref{thm: Bourgain}.




\subsection*{Acknowledgement} The second author is supported by New Investigator Award UKRI097. Much of this work was carried out during summer 2025 when the first author was supported by EPSRC's Vacation Internships scheme at the University of Edinburgh. The authors are grateful to David Beltran for various suggestions which improved the exposition, and to Tony Carbery for help with the referencing. 




\section{Necessary conditions}\label{sec: nec}

For $d \geq 2$ and $1 \leq p <\infty$, suppose the estimate 
\begin{equation}\label{eq: test ineq}
  \|M f\|_{L^p(\R^d)} \lesssim_p \|f\|_{L^p(\R^d)}  
\end{equation}
holds for all $f \in C_c(\R^d)$. By density, $M$ can then be extended to an operator on $L^p(\R^d)$ which satisfies \eqref{eq: test ineq} for all $f \in L^p(\R^d)$. We test this estimate against some simple examples. 

\begin{example}[Dimensionality condition]\label{ex: dim} For $d \geq 2$ and $0 < \delta \leq 1$, we test \eqref{eq: test ineq} against the function $f \coloneq \chi_{B(0,\delta)}$. For each $x \in \R^d$ with $1 \leq |x| \leq 2$, we choose $r = |x|$, noting that the sphere $C(x,|x|)$ passes through the origin. It then follows that $A \chi_{B(0,\delta)} (x,|x|) \gtrsim \delta^{d-1}$ for all $1 \leq |x| \leq 2$. Thus, if \eqref{eq: test ineq} were to hold, this would force 
\begin{equation*}
    \delta^{d-1} \lesssim \|M \chi_{B(0,\delta)}\|_{L^p(\R^d)} \lesssim_p |B(0,\delta)|^{1/p} \sim \delta^{d/p}, \quad 0 < \delta \leq 1.
\end{equation*}
This implies $p \geq \frac{d}{d-1}$, which shows that Theorem~\ref{thm: Stein} and Theorem~\ref{thm: Bourgain} are sharp, at least up to the endpoint. 
\end{example}

We remark that a variant of the above example can be used to rule out boundedness at the $p = \frac{d}{d-1}$  endpoint exponent.\footnote{In particular, one takes $f(y) \coloneq |y|^{-(d-1)} \log(1/|y|)^{-1} \chi_{B(0,1/2)}(y)$, which can be thought of as a weighted superposition of the ball example at different scales. See, for instance, \cite[p.472]{SteinBook}.}

As the naming suggests, Example~\ref{ex: dim} is sensitive to the dimension of the sphere relative to the ambient Euclidean space. In addition, we consider the following alternative example which is sensitive to the curvature of the sphere. 

\begin{example}[Curvature condition]\label{ex: curv} For $d \geq 2$ and $0 < \delta \leq 1$, let $R(\delta) \coloneq [-\delta^{1/2}, \delta^{1/2}]^{d-1} \times [-\delta, \delta]$. We test \eqref{eq: test ineq} against the function $f \coloneq \chi_{R(\delta)}$. For all $x \in [-\delta^{1/2}, \delta^{1/2}]^{d-1} \times [1,2]$, we choose $r = x_n$, noting that $C(x,x_n)$ is tangent to the $x_n = 0$ hyperplane. It then follows that $A \chi_{R(\delta)} (x,x_n) \gtrsim \delta^{(d-1)/2}$ for all $x \in [-\delta^{1/2}, \delta^{1/2}]^{d-1} \times [1,2]$.  Thus, if \eqref{eq: test ineq} were to hold, this would force 
\begin{equation*}
    \delta^{(d-1)/2} \delta^{(d-1)/(2p)}  \lesssim \|M \chi_{B(0,\delta)}\|_{L^p(\R^d)} \lesssim_p |R(\delta)|^{1/p} \sim \delta^{(d+1)/(2p)}, \quad 0 < \delta \leq 1.
\end{equation*}
This implies $p \geq \frac{2}{d-1}$, which for $d \geq 3$ is a much weaker necessary condition than that arising from Example~\ref{ex: dim}. However, for $d = 2$, we obtain $p \geq 2$, which matches our earlier condition. 
\end{example}

For $d = 2$ a variant of the above example can be used to rule out boundedness of $M$ at the $p = 2$ endpoint exponent.

From the above, we see that Example~\ref{ex: curv} provides an alternative sharp example (at least up to the endpoint) for $L^p$ boundedness of $M$ in the $d=2$ case, but does not provide a sharp example for $d \geq 3$. This already hints at possible differences and additional complications in the study of the circular maximal function, compared with its higher dimensional counterpart. We remark that, by taking many rotated and translated copies of Example~\ref{ex: curv}, it is possible to show concrete differences between the endpoint mapping properties of the circular maximal function and the spherical maximal function for $d \geq 3$: see \cite{STW2003}.




\section{Functional and Fourier analytic preliminaries}\label{sec: Fourier prelim}




\subsection{Mollification}\label{subsec: mollification} The spherical means \eqref{eq: spherical mean} are singular, in the sense that the integration takes place over a lower-dimensional set. We follow a standard scheme in analysis by replacing this singular operator with a sequence of non-singular mollified operators. For this, we perform a standard Littlewood--Paley decomposition.

Fix $\eta \in C_c^\infty (\R)$ even such that
\begin{equation*}
   \eta(s) = 1 \quad \textrm{if $|s| \leq 1$} \qquad \text{ and} \qquad
   \eta(s) = 0 \quad \text{if $|s| \geq 2.$}
\end{equation*}
For $j \in \N_0$, define $\beta_j \in C^{\infty}_c(\R^d)$ by
\begin{equation*}
    \beta_0(\xi) \coloneq \eta(|\xi|) \quad \textrm{and} \quad \beta_j(\xi) \coloneq \beta(2^{-j}\xi), \quad j \in \N, \quad \textrm{for} \quad \beta(\xi) \coloneq \beta_0(\xi)-\beta_0(2\xi).
\end{equation*}
 It then follows that
 \begin{equation*}
   \supp \beta_j \subseteq \{\xi \in \widehat{\R}^d :  2^{j-1} \leq |\xi| \leq 2^{j+1}\}, \quad   j \in \N, \quad \textrm{and} \quad 1 = \sum_{j=0}^\infty \beta_j(\xi), \quad \xi \in \widehat{\R}^d.
 \end{equation*}
For $r > 0$, we decompose $Af$ as a (pointwise) sum of frequency-localised operators
\begin{equation*}
    Af = \sum_{j=0}^\infty A^jf \quad \textrm{where} \quad A^jf(x, r) \coloneq \int_{\widehat{\R}^d} e^{2\pi i x \cdot \xi} \, \beta_j(r \xi) \, \hat{\sigma}_r(\xi) \, \hat{f}(\xi) \, \ud \xi \quad \textrm{for $j \in \N_0$.}
\end{equation*}
Alternatively,
\begin{equation*}
  A^jf(x, r) = f*(\varphi_{j,r}*\sigma_r)(x)  \quad \textrm{where} \quad \varphi_{j,r}(x) \coloneq r^{-d}(\beta_j)\;\widecheck{}\;(r^{-1} x) \quad \textrm{for $j \in \N_0$.}
\end{equation*}
Note that the convolution kernel $\varphi_{j,r}*\sigma_r$ is a Schwartz function, rather than a singular measure.




\subsection{ Oscillatory kernels}\label{sec: osc kernels} We turn to studying the kernel $\varphi_{j,r}*\sigma_r$ of $A^jf(\,\cdot\,, r)$. In what follows, it is useful to adopt the following notation.

\begin{definition}[$\delta$-annulus]\label{dfn: annulus} Given a sphere $C = C(x,r) \subseteq \R^d$, where $(x,r) \in \R^d \times (0, \infty)$, and $\delta > 0$, we define the \textit{$\delta$-annulus}
\begin{equation*}
    C^{\delta} = C^{\delta}(x,r) \coloneq \big\{ y \in \R^d : ||y - x| - r| < \delta\big\}.
\end{equation*}
Thus, $C^{\delta}$ is simply the $\delta$-neighbourhood of $C$. 
\end{definition}

\begin{definition}[Oscillatory kernel]\label{dfn: Chi} Given $0 < \delta \leq 1/2$ dyadic and $C = C(x,r)$ for $(x,r) \in \R^d \times (0,\infty)$, define the \textit{oscillatory kernel} $\Chi_{C^\delta} \colon \R^d \to \C$ by
\begin{equation}\label{eq: Chi}
    \Chi_{C^{\delta}} (y) \coloneq \delta \cdot \varphi_{j,r} * \sigma_r (x - y) \qquad \textrm{for all $y \in \R^d$,}
\end{equation}
where $j \in \N$ satisfies $\delta = 2^{-j}$.
\end{definition}
Given $j \in \N$ and $(x,r) \in \R^d \times (0,\infty)$, we may therefore write
\begin{equation}\label{eq: Chi in action}
    A^jf(x,r) = \delta^{-1} \int_{\R^d} f(y) \Chi_{C^{\delta}(x,r)} (y) \,\ud y \qquad \textrm{for $\delta \coloneq 2^{-j}$.}
\end{equation}

If $r \in [1,2]$, then $\varphi_{j,r}$ has Fourier support at scale $|\xi| \sim 2^j$. Thus, by the uncertainty principle, we expect $\Chi_{C^\delta}$ to be essentially supported on ${C^\delta}$; indeed, this is the motivation for the notation \eqref{eq: Chi}. The following lemma makes this precise. 

\begin{lemma}[Essential support] \label{lem:essentialsupp}
  Let $C = C(x,r)$ for $(x,r) \in \R^d \times [1,2]$ and $0 < \delta \leq 1/2$ be dyadic. If $y\in \R^d$ satisfies $\dist(y, C) \geq \delta^{1-\eta}$ for $\eta > 0$, then 
    \begin{equation}
          |\Chi_{C^\delta} (y)| \lesssim_{N, \eta} \delta^N (1+|y-x|)^{-10d} \quad \text{for all $N\in \N_0$.}
        \label{eq:esssup1}
    \end{equation}
\end{lemma}

Lemma~\ref{lem:essentialsupp} shows that $\Chi_{C^\delta}$ is tiny outside any slight enlargement of $C^{\delta}$, and so the function is in this sense `essentially supported' on $C^{\delta}$.

\begin{proof}[Proof (of Lemma~\ref{lem:essentialsupp})] Let $\varphi_r (x) \coloneq r^{-d}\check{\beta}(r^{-1} x)$. Recall from definition \eqref{eq: Chi} that
    \begin{equation*}
        \Chi_{C^\delta} (y) \coloneq \delta \sigma_r * \varphi_{j,r}(x-y)= \delta \cdot \delta^{-d} \int_{S^{d-1}} \varphi_r(\delta^{-1} (x-y-r\omega)) \ud \sigma(\omega),
    \end{equation*}
    where $\delta = 2^{-j}$. Since $\varphi_r$ is Schwartz with uniform decay for $r \in [1,2]$, we have
    \begin{equation}\label{eq:esssup2}
        |\Chi_{C^\delta}(y)| \lesssim_N \delta^{1-d} \int_{S^{d-1}} (1 + \delta^{-1} |x - y - r \omega|)^{-N} \ud \sigma(\omega) \quad \text{for all $N \in \N_0$.}
    \end{equation}

     Suppose $y \in \R^d$ satisfies $\dist(y, C) \geq \delta^{1- \eta}$ for $\eta > 0$, so that $|x-y-r\omega| \geq \delta^{1-\eta}$ for all $\omega \in S^{d-1}$. Applying this to \eqref{eq:esssup2}, we deduce that $|\Chi_{C^\delta}(y)| \lesssim_{N, \eta} \delta^N$ for all $N \in \N_0$. This establishes the desired bound \eqref{eq:esssup1} in the case $|y-x| < 4$. Finally, if $|y-x| \geq 4$, then we may bound $|x-y-r\omega| \geq |y-x|/2$. Applying this to \eqref{eq:esssup2}, in conjunction with our earlier observations, we deduce that \eqref{eq:esssup1} holds in the remaining case.
\end{proof}

Due in part to our choice of normalisation, we also heuristically expect the (non-rigorous) identity $|\Chi_{C^\delta}| = \chi_{C^\delta}$. One rigorous manifestation of this heuristic is the following lemma.

\begin{lemma}\label{lem: Lp Chi} Let $0 < \delta \leq 1/2$ be dyadic and $C = C(x,r)$ for $(x,r) \in \R^d \times [1,2]$. For all $1 \leq p \leq \infty$ we have
\begin{equation*}
    \|\Chi_{C^{\delta}}\|_{L^p(\R^d)} \lesssim \delta^{1/p}. 
\end{equation*}
\end{lemma}

\begin{proof} 

By log-convexity of $L^p$, it suffices to consider $p = 1$ and $p = \infty$ only.\medskip

\noindent \underline{$p = 1$}. Recall that $\Chi_{C^\delta} (y) \coloneq \delta \cdot \varphi_{j,r} * \sigma_r (x - y)$, where $j \in \N$ satisfies $\delta = 2^{-j}$. Thus, by Young's inequality,
\begin{equation*}
    \norm{\Chi_{C^\delta}}_{L^1(\R^d)} 
    = 
    \delta \, \norm{\varphi_{j,r} * \sigma_r}_{L^1(\R^d)} 
    \lesssim
    \delta \, \norm{\varphi_{j,r}}_{L^1(\R^d)} \, \norm{\sigma_r} 
    \lesssim
    \delta,
\end{equation*}
as required.\medskip

\noindent \underline{$p = \infty$}. For $\varphi_r (x) \coloneq r^{-d}\check{\beta}(r^{-1} x)$ as in the proof of the previous lemma,
\begin{equation*}
    \norm{\Chi_{C^\delta}}_{L^{\infty}(\R^d)} 
    = 
    \delta \, \norm{\varphi_{j,r} * \sigma_r}_{L^{\infty}(\R^d)} =  \delta^{1-d} \sup_{y \in\R^d} \int_{S^{d-1}} |\varphi_r(\delta^{-1}(y-r\omega))| \, \ud \sigma(\omega).
\end{equation*} 
    Ignoring Schwartz tails, $\varphi_r$ can be thought of as concentrated on the unit ball $B(0,1)$. The above integral is then bounded by the maximum $\sigma_r$-measure of the intersection between $rS^{d-1}$ and a ball $B(y,\delta)$. This clearly gives the desired bound. It is not difficult to make this argument rigorous, dealing with the Schwartz tails via dyadic decomposition (see, for example, \cite[Lemma 6.5.3]{Grafakos_book}).
\end{proof}




\subsection{The local maximal operator} 

For $j \in \N_0$ and the frequency localised operators $A^j$ as defined in \S\ref{subsec: mollification}, we introduce the \textit{local} maximal operators
\begin{equation}\label{eq: loc max}
  M_{\ell}f(x) \coloneq \sup_{2^{\ell} \leq r \leq 2^{\ell+1}} |Af(x, r)| \quad \textrm{and} \quad    M^j_{\ell}f(x) \coloneq \sup_{2^{\ell} \leq r \leq 2^{\ell+1}} |A^jf(x, r)|, \quad \ell \in \Z.
\end{equation}
Note that the supremum is now restricted to the compact interval $[2^{\ell},2^{\ell+1}]$, rather than the entire range $(0, \infty)$ in the definition of $Mf$. 

\begin{lemma}\label{lem: loc to glob} Let $2 \leq p < \infty$ and suppose there exists some $\varepsilon(p) > 0$ such that
\begin{equation}\label{eq: loc to glob a}
    \| M^j_0f\|_{L^p(\R^d)} \lesssim 2^{-j\varepsilon(p)} \|f\|_{L^p(\R^d)}
\end{equation}
holds for all $j \in \N$. Then 
\begin{equation}\label{eq: loc to glob b}
    \|M f\|_{L^p(\R^d)} \lesssim \|f\|_{L^p(\R^d)}.
\end{equation}
\end{lemma}

\begin{proof} Fixing $j \in \N_0$, it suffices to show
\begin{equation}\label{eq: loc to glob 1}
   \|M^jf\|_{L^p(\R^d)} \lesssim 2^{-j\varepsilon(p)} \|f\|_{L^p(\R^d)} \quad \textrm{where} \quad M^jf(x) \coloneq \sup_{r > 0} |A^jf(x,r)|, 
\end{equation}
since then \eqref{eq: loc to glob b} immediately follows by summing together the different frequency contributions using the geometric decay factor. For $j = 0$, we may pointwise bound $M^0f(x) \lesssim M_{\textrm{HL}}f(x)$ where $M_{\textrm{HL}}$ denotes the Hardy--Littlewood maximal function (see, for instance, \cite[Ch. II, \S2.1, Proposition]{SteinBook}). This immediately implies \eqref{eq: loc to glob 1} in this case.

Henceforth, assume $j \in \N$. We may pointwise bound $M^j$ by an $\ell^p$ sum of the $M^j_{\ell}$, giving
\begin{equation}\label{eq: loc to glob 2}
        \|M^j f\|_{L^p(\R^d)} = \|\sup_{k \in \Z} M^j_{j - k}f\|_{L^p(\R^d)}  \leq \Brac{\sum_{k \in \Z} \| M^j_{j - k}f\|_{L^p(\R^d)}^p}^{1/p}.
    \end{equation}
Note that, by a simple scaling argument, \eqref{eq: loc to glob a} automatically implies
\begin{equation}\label{eq: loc to glob 3}
    \|M^j_{\ell}f\|_{L^p(\R^d)} \lesssim 2^{-j\varepsilon(p)} \|f\|_{L^p(\R^d)} \qquad \textrm{for all $\ell \in \Z$.}  
\end{equation}
This allows us to bound the individual terms in \eqref{eq: loc to glob 2}, but we still need to carry out the summation in $k$. For this, we appeal to Fourier orthogonality.

Let $\tilde{\beta} \in C^{\infty}_c(\widehat{\R}^d)$ be a bump function satisfying $\tilde{\beta}(\xi) = 1$ if $1/4 \leq |\xi| \leq 4$ and $\tilde{\beta}(\xi) = 0$ if $|\xi| \notin [1/8,8]$. Define the associated Littlewood--Paley projectors $P_k$ by
 \begin{equation*}
     (P_k f) \ft (\xi) \coloneq \tilde{\beta} (2^{-k} \xi) \hat{f} (\xi) \qquad \textrm{for all $k \in \Z$.}
 \end{equation*}
Since $2 \leq p < \infty$, by the classical Littlewood--Paley inequality,\footnote{The leftmost expression can also be directly estimated by the rightmost expression by interpolation between the (elementary) $p=2$ and $p=\infty$ cases.}
\begin{equation}\label{eq: loc to glob 4}
    \Brac{\sum_{k \in \Z} \|P_kf\|_{L^p(\R^d)}^p}^{1/p} \lesssim \Big\| \Big( \sum_{k \in \Z} |P_kf|^2 \Big)^{1/2}\Big\|_{L^p(\R^d)}\lesssim \|f\|_{L^p(\R^d)}.
\end{equation}

Recalling the properties of $\beta$ as defined in \S\ref{subsec: mollification}, it is a simple matter to verify
    \begin{equation*}
        A^jf(x,r) = A^j (P_{j-\ell} f)(x,r) \qquad 
        \textrm{for $2^\ell \leq r \leq 2^{\ell + 1}$.}
    \end{equation*}
Combining this with \eqref{eq: loc to glob 3}, we deduce that
   \begin{equation}\label{eq: loc to glob 5}
    \|M^j_{\ell}f\|_{L^p(\R^d)} \lesssim 2^{-j\varepsilon(p)} \|P_{j-\ell}f\|_{L^p(\R^d)} \qquad \textrm{for all $\ell \in \Z$.}  
\end{equation}
We apply \eqref{eq: loc to glob 5} to each term on the right-hand side of \eqref{eq: loc to glob 2}. Finally, applying the Littlewood--Paley inequality \eqref{eq: loc to glob 4} to the resulting expression yields the desired bound \eqref{eq: loc to glob 1} and thereby concludes the proof. 
\end{proof}

By restricting the radii to $r \in [1,2]$, the operator $M^j_0$ behaves essentially locally, in the sense that the value $M^j_0f(x)$ is essentially determined by the values of $f(y)$ for, say, $y \in B(x, 2)$.\footnote{This is not literally true, due to the presence of Schwartz tails in the kernel of $A^j$. Nevertheless, values of $f(y)$ for $y \notin B(x, 2)$ are attributed very little weight in our averages.} This leads to the following (standard) reduction. 

\begin{lemma}\label{lem: loc op red}
Let $1 \leq p < \infty$ and $j \in \N$ and suppose the bound
\begin{equation}\label{eq: loc op red a}
    \| M^j_0f\|_{L^p(\bbB^d)} \lesssim 2^{-j \varepsilon(p)} \|f\|_{L^p(\R^d)}
\end{equation}
holds for $\bbB^d \coloneq B(0, 1/2)$ and some $\varepsilon(p) \in \R$ with $\varepsilon(p) \leq d$. Then 
\begin{equation*}
    \|M^j_0 f\|_{L^p(\R^d)} \lesssim 2^{-j \varepsilon(p)} \|f\|_{L^p(\R^d)}.
\end{equation*}
\end{lemma}

\begin{remark}\label{rmk: no ext tang} A consequence of the above localisation is that it rules out \textit{externally} tangent pairs of spheres. In particular, if $(x_1, r_1)$, $(x_2, r_2) \in \bbB^d \times [1,2]$, then 
\begin{equation*}
r_1 + r_2 - |x_1 - x_2| \geq 2 - 1 = 1 > 0,
\end{equation*} 
which means $C(x_1, r_1)$ and $C(x_2, r_2)$ cannot be externally tangent (and are in some sense always quantitatively `far' from being externally tangent). However, it is still possible that $|r_1 - r_2| = |x_1 - x_2|$, in which case $C(x_1, r_1)$ and $C(x_2, r_2)$ are \textit{internally} tangent. 
\end{remark}

\begin{proof}[Proof (of Lemma~\ref{lem: loc op red})] Let $\cB$ be a boundedly overlapping cover of $\R^d$ by balls of radius $1/2$. By translation invariance, \eqref{eq: loc op red a} holds with $\bbB^d$ replaced with any $B \in \cB$.

For $B \in \cB$, let $B^*$ denote the ball concentric to $B$ but with radius $7/2$. We bound
\begin{equation}\label{eq: loc op red 1}
    \|M_0^jf\|_{L^p(B)} \leq \big\|M_0^j\big(\chi_{B^*} f\big)\big\|_{L^p(B)} + \big\|M_0^j\big(\chi_{\R^d \setminus B^*} f\big)\big\|_{L^p(B)}.
\end{equation}

By the local estimate \eqref{eq: loc op red a} and the bounded overlap of $\{B^* : B \in \cB\}$, we have
\begin{equation}\label{eq: loc op red 2}
    \Big(\sum_{B \in \cB} \big\|M^j_0\big(\chi_{B^*} f\big)\big\|_{L^p(B)}^p\Big)^{1/p} \lesssim 2^{-j \varepsilon(p)} \Big(\sum_{B \in \cB} \| f\|_{L^p(B^*)}^p\Big)^{1/p} \lesssim 2^{-j \varepsilon(p)} \|f\|_{L^p(\R^d)}. 
\end{equation}
On the other hand, fixing $B \in \cB$, if $x \in B$, $\omega \in S^{d-1}$, $1 \leq r \leq 2$ and $y \in \R^d \setminus B^*$, then $|x + r \omega - y| \geq 1$. In particular, $\dist(y, C(x,r)) \geq \delta^{1- \eta}$ for $\delta \coloneq 2^{-j}$ and $\eta \coloneq 1$. Thus, by Lemma~\ref{lem:essentialsupp}, we have
\begin{equation*}
    |\Chi_{C^{\delta}(x,r)}(y)| \lesssim \delta^d (1 + |x-y|)^{-10d} \leq 2^{-j\varepsilon(p)}  K(x-y) \quad \textrm{for} \quad K(y) \coloneq (1 + |y|)^{-10d}. 
\end{equation*}
Therefore, by Young's inequality and the bounded overlap of $\cB$, we conclude that
\begin{equation}\label{eq: loc op red 3}
    \Big(\sum_{B \in \cB} \big\|M^j_0\big(\chi_{\R^d \setminus B^*} f\big)\big\|_{L^p(B)}^p\Big)^{1/p} \lesssim 2^{-j \varepsilon(p)} \| K \ast |f|\|_{L^p(\R^d)} \lesssim 2^{-j\varepsilon(p)}  \|f\|_{L^p(\R^d)}. 
\end{equation}
Taking the $\ell^p$ sum of \eqref{eq: loc op red 1} over all $B \in \cB$ and combining the resulting bound with \eqref{eq: loc op red 2} and \eqref{eq: loc op red 3}, we obtain the desired estimate. 
\end{proof}




\subsection{Discretisation, linearisation and duality} In light of Lemmas~\ref{lem: loc op red} and~\ref{lem: loc to glob}, to prove the circular maximal theorem for $d = 2$ and $p > 2$, matters are reduced to proving estimates of the form \eqref{eq: loc op red a}. 

By the uncertainty principle, since the averages $A^jf$ are frequency localised at scale $2^j$, we expect $|A^jf|$ to be essentially constant at scale $2^{-j}$. This motivates breaking the spatial domain into cubes of side-length $2^{-j}$.

Fix $j \in \N$ and let $\cQ_j$ denote the minimal covering of $\bbB^d$ by closed dyadic cubes of sidelength $2^{-j}$, where $\bbB^d$ is as defined in Lemma~\ref{lem: loc op red}. For each $Q_j \in \cQ_j$, we can find some $x_{Q_j} \in Q_j$ and $r_{Q_j} \in [1,2]$ close to realising the relevant suprema, so that
\begin{equation*}
     M^j_0f(x) \leq 2  M^j_0 f(x_{Q_j}) \leq 4 A^j f(x_{Q_j}, r_{Q_j}) \qquad \textrm{for all $x \in Q_j$.}
\end{equation*}
Write $C_{Q_j} = C(x_{Q_j}, r_{Q_j})$ and let $\delta \coloneq 2^{-j}$. Using \eqref{eq: Chi in action}, each $A^j f(x_{Q_j}, r_{Q_j})$ is given by a normalised integral of $f$ against the associated oscillatory kernel $\Chi_{C_{Q_j}^{\delta}}$. Decomposing the $L^p(\bbB^d)$-norm as an $\ell^p$ sum of norms over the constituent cubes of $\cQ_j$, we therefore see that
\begin{equation*}
 \norm{ M^j_0f}_{L^p(\bbB^d)}^p \lesssim \sum_{Q_j \in \cQ_j} \delta^d \cdot \Abs{ \delta^{-1}\int_{\R^d} f(y) \Chi_{C_{Q_j}^{\delta}}(y)\,\ud y}^p.
\end{equation*}
Thus, by relabelling and a simple pigeonholing argument,\footnote{That is, we partition $\{C_{Q_j} : Q_j \in \cQ_j\}$ into $O(1)$ subsets, each formed of spheres with $\delta$-separated centres, and define $\cC$ to be a subset which maximises the right-hand side of~\eqref{eq: duality 1}.} 
\begin{equation}\label{eq: duality 1}
    \norm{ M^j_0f}_{L^p(\bbB^d)} \lesssim \delta^{d/p - 1} \Brac{ \sum_{C \in \cC} \Abs{\int_{\R^d} f(y) \Chi_{C^\delta}(y)\,\ud y}^p}^{1/p}
\end{equation}
where $\cC$ is a set of spheres centred at $\delta$-separated points in $\bbB^d$, with radii lying in $[1,2]$. These observations prompt the following definition.

\begin{definition}
Let $d \geq 2$ and $0 < \delta \leq 1/2$.
\begin{enumerate}[i)]
    \item A sphere $C$ in $\R^d$ is \textit{unit scale} if $C = C(x,r)$ for some $x \in \bbB^d$ and $r \in [1,2]$.
    \item A set $\cC$ of spheres in $\R^d$ is \textit{$\delta$-separated} if $|x_1 - x_2| \geq \delta$ whenever $C(x_1, r_1)$, $C(x_2, r_2) \in \cC$ are distinct. 
\end{enumerate} 
\end{definition}

Combining \eqref{eq: duality 1} with duality of $\ell^p$, we are led to the following reformulation of the maximal estimate. 

\begin{lemma}[Duality]\label{lem: duality}
    To prove the local maximal bound \eqref{eq: loc op red a} for fixed $1 \leq p < \infty$ and $\varepsilon(p) \in \R$, it suffices to show
    \begin{equation}\label{eq: duality}
        \NormLpRd{\sum_{C \in \cC} a_C \Chi_{C^\delta}}{p'}{d} \lesssim \delta^{1-d/p + \varepsilon(p)} \Brac{ \sum_{C \in \cC} |a_C|^{p'}}^{1/p'}
    \end{equation}
    holds for all $0 < \delta \leq 1/2$ dyadic, $\cC$ any $\delta$-separated set of unit scale spheres and $(a_C)_{C \in \cC}$ any complex sequence. 
\end{lemma}

Given \eqref{eq: duality 1}, the proof of Lemma~\ref{lem: duality} is a straightforward exercise. We remark that such duality arguments date back to C\'ordoba~\cite{Cordoba1977} and are standard within the geometric maximal function literature; see also \cite{Carbery1988} or \cite[Proposition 22.4]{Mattila_book} and \cite[Proposition 22.6]{Mattila_book} for an exposition of the argument in the Kakeya case, which easily adapts to the present setting.




\subsection{Comparison with geometric arguments}\label{subsec: geom comparison} We end this section by briefly describing purely geometric counterparts of the above arguments, which do not rely on Fourier analysis. The material in this subsection is not needed for our proof of the circular maximal theorem. We nevertheless include it, since the purely geometric perspective sets up an effective foil for elucidating the role of the Fourier transform in our later arguments.\medskip

\noindent \textit{Mollification}. An alternative, arguably more direct, approach to mollifying the operator is to simply note
\begin{equation*}
    Af(x, r) = \lim_{\delta \rightarrow 0_+} \cA^{\delta}f(x, r) \qquad \textrm{where} \qquad \cA^{\delta}f(x, r) \coloneq \frac{1}{|C^{\delta}(x,r)|}\int_{C^\delta (x,r)} f
\end{equation*}
for $C^{\delta}(x,r)$ as in Definition~\ref{dfn: annulus}. Note that the integration in  $\cA^{\delta}f$ now takes place over a $d$-dimensional set.\medskip

\noindent \textit{The local maximal operator}. For $0 < \delta \leq 1/2$, we introduce the \textit{local} maximal operators
\begin{equation*}
\cM_0^{\delta} f(x) \coloneq \sup_{1 \leq r \leq 2} |\cA^{\delta} f(x, r)|.
\end{equation*}
Our goal now is to prove estimates of the form
\begin{equation}\label{eq: geom loc}
    \|\cM_0^{\delta} f\|_{L^p(\R^d)} \lesssim \|f\|_{L^p(\R^d)},
\end{equation}
with a constant that is uniform in $0 < \delta \leq 1/2$. Once \eqref{eq: geom loc} is established, a simple limiting argument implies the corresponding $L^p$-boundedness of $M_0$, where $M_0$ is the local maximal operator defined in \eqref{eq: loc max}. Here we only consider local maximal operators. This is because, in contrast with the Fourier analytic approach, there is no direct mechanism to combine bounds for the local operator $\cM_0^{\delta}$ to bound the global operator $\cM$.

The key difference between \eqref{eq: geom loc} and its counterpart \eqref{eq: loc to glob a} in the Fourier analytic approach is that here we no longer require the extra decay factor $2^{-j\varepsilon(p)} = \delta^{\varepsilon(p)}$ in our estimates. This is because we are now interested in taking a limit in the $\delta$ parameter, rather than summing a series in the $j$ parameter. Indeed, \eqref{eq: geom loc} cannot possibly hold with an additional $\delta^{\varepsilon(p)}$ factor on the right-hand side, for $\varepsilon(p) > 0$, since otherwise taking the limit would force $\cM_0 f$ to vanish almost everywhere for every choice of $f \in C_c(\R^d)$. \medskip

\noindent \textit{Discretisation, linearisation and duality}. By using parallel arguments, we can establish the following geometric analogue of the duality described in Lemma~\ref{lem: duality}.

\begin{lemma}[Geometric duality]\label{lem: geom duality}
    To prove the local maximal bound \eqref{eq: geom loc} for some fixed $1 \leq p < \infty$, it suffices to show
    \begin{equation}\label{eq: geom duality}
        \NormLpRd{\sum_{C \in \cC} a_C \chi_{C^\delta}}{p'}{d} \lesssim \delta^{1-d/p} \Brac{ \sum_{C \in \cC} |a_C|^{p'}}^{1/p'}
    \end{equation}
    holds for all $0 < \delta \leq 1/2$, any $\delta$-separated set $\cC$ of unit scale spheres and $(a_C)_{C \in \cC}$ any complex sequence. 
\end{lemma}

Once again, the key difference between Lemma~\ref{lem: geom duality} and its Fourier analytic counterpart Lemma~\ref{lem: duality} is that here we no longer require the extra decay factor $\delta^{\varepsilon(p)}$. Using the essential support properties of the $\Chi_{C^{\delta}}$, it is not difficult to show that any bound \eqref{eq: geom duality} implies \eqref{eq: duality} with $\varepsilon(p) = 0$: that is, \eqref{eq: geom duality} implies the same estimate but with the $\chi_{C^{\delta}}$ replaced with the oscillatory functions $\Chi_{C^{\delta}}$. In order to prove \eqref{eq: duality} with some $\varepsilon(p) > 0$, however, one must exploit additional cancellation due to the oscillation of the $\Chi_{C^{\delta}}$.




\section{$L^2$ theory and Stein's theorem}\label{sec: Stein}




\subsection{Frequency localised $L^2$ bound} Here we shall carefully examine the oscillatory behaviour of the $\Chi_{C^{\delta}}$ in order to prove $L^2$ estimates for the frequency localised maximal function. Defining the exponent 
\begin{equation}\label{eq: L2 exp}
\varepsilon_d(2) \coloneq \frac{d}{2}-1,    
\end{equation}
the main result of this section reads thus.  

\begin{proposition}[Frequency localised $L^2$ bound]\label{prop:steinlpdualp2}
    For all $d \geq 2$, the inequality
    \begin{equation}
        \NormLpRd{\sum_{C \in \cC} a_C \Chi_{C^\delta}}{2}{d} \lesssim \delta^{1-d/2 + \varepsilon_d(2)} \Brac{ \sum_{C \in \cC} |a_C|^{2}}^{1/2} \label{eq:steinlpdualp2} 
    \end{equation}
    holds for all $0 < \delta \leq 1/2$ dyadic, $\cC$ any $\delta$-separated set of unit scale spheres and $(a_C)_{C \in \cC}$ any complex sequence. 
\end{proposition}

This bound is essentially contained in Stein's spherical maximal paper \cite{Stein1976}, albeit in a rather different guise. However, our proof of Proposition~\ref{prop:steinlpdualp2} involves somewhat different methods from those featured in \cite{Stein1976}. 

For $d \geq 3$, we have $\varepsilon_d(2) > 0$ and so Proposition~\ref{prop:steinlpdualp2} can be combined with Lemma~\ref{lem: duality}, Lemma~\ref{lem: loc op red} and Lemma~\ref{lem: loc to glob} to deduce the $L^2$ boundedness of the spherical maximal operator. For $d = 2$, however, we have $\varepsilon_2(2) = 0$ and so the frequency localised pieces fail to be summable (as we should expect, since we know from \S\ref{sec: nec} that the circular maximal function is not $L^2$ bounded). Nevertheless, Proposition~\ref{prop:steinlpdualp2} is, in some sense, very close to proving $L^2$ boundedness of the circular maximal theorem: the condition $\varepsilon_2(2) = 0$ means that summability only just fails. Indeed, we shall see that Proposition~\ref{prop:steinlpdualp2} forms the foundation for Bourgain's proof of Theorem~\ref{thm: Bourgain}. 

Proposition~\ref{prop:steinlpdualp2} rests on certain orthogonality properties of the $\Chi_{C^\delta}$. To state these properties, we introduce the following definition.

\begin{definition} 
For $(x_i, r_i) \in \R^d \times (0,\infty)$ and $C_i \coloneq C(x_i, r_i)$ for $i = 1$, $2$, define
\begin{equation*}
 \dist(C_1, C_2) \coloneq  |x_1 - x_2| \qquad \textrm{and} \qquad  \tang{C_1}{C_2} \coloneq ||x_1-x_2|-|r_1-r_2||.
\end{equation*}
\end{definition}

Note that $\tang{C_1}{C_2} = 0$ if and only if $C_1$ and $C_2$ are interior tangent. Thus, in general, $\tang{C_1}{C_2}$ can be thought of as a measure of how close $C_1$ and $C_2$ are to being interior tangency (see also Figure~\ref{fig: omega} below). 

\begin{lemma}[Weak orthogonality]\label{lem:weakorthogonality} Fix $C_1$, $C_2 \subset \R^d$ unit scale spheres and $0 < \delta \leq 1/2$ dyadic. For all $N\in\N_0$, we have
\begin{equation}
    |\inn{\Chi_{C_1^\delta}}{\Chi_{C_2^\delta}}| \lesssim_N \delta \cdot (1 + \delta^{-1} \dist\big(C_1, C_2)\big)^{-(d-1)/2} \big(1+\delta^{-1} \tang{C_1}{C_2}\big)^{-N}. \label{eq: w orthog}
\end{equation}
\end{lemma}

\begin{remark} For Lemma~\ref{lem:weakorthogonality} to hold, it is important that we work with \textit{unit scale} spheres, which rules out the possibility that the spheres are externally tangent: see Remark~\ref{rmk: no ext tang}. 
\end{remark}

Bounds of this form have appeared in, for instance, \cite[Lemma 3.2]{Cladek2018} and also \cite[Lemma 3.3]{HNS2011}. Lemma~\ref{lem:weakorthogonality} tells us that $\inn{\Chi_{C_1^\delta}}{\Chi_{C_2^\delta}}$ is essentially negligible unless $C_1$ and $C_2$ are almost tangent, in the sense that $\tang{C_1}{C_2} \lesssim \delta$. This property heavily exploits the oscillation of the $\Chi_{C^\delta}$, and is not merely a consequence of the essential support properties described in \S\ref{sec: osc kernels}. We postpone the proof of Lemma~\ref{lem:weakorthogonality} until \S\ref{subsec: w orthog} below; presently, we observe how Lemma~\ref{lem:weakorthogonality} can be used to establish Proposition~\ref{prop:steinlpdualp2}.

\begin{proof}[Proof (of Proposition~\ref{prop:steinlpdualp2})] Let $(I_k)_{k=1}^K$ be a partition of $[1,2]$ into $K \sim \delta^{-1}$ intervals of length at most $\delta/2$ and define
\begin{equation*}
    \cC_k \coloneq \big\{ C\in\cC : C = C(x,r) \textrm{ for some }  (x,r) \in \bbB^d \times I_k \big\} \qquad \textrm{for $1 \leq k \leq K$.}
\end{equation*}
By the Cauchy--Schwarz inequality,
    \begin{equation}
        \NormLpRd{\sum_{C \in \cC} a_C\Chi_{C^\delta}}{2}{d}
        \lesssim \delta^{-1/2} \Brac{\sum_{k=1}^K \NormLpRd{\sum_{C \in \cC_k} a_C\Chi_{C^\delta}}{2}{d}^2 \, }^{1/2}. \label{eq:l2freqlocproof2}
    \end{equation}

    Fix $1 \leq k \leq K$ and write
    \begin{equation}\label{eq:l2freqlocproof3}
        \NormLpRd{\sum_{C \in \cC_k} a_C \Chi_{C^\delta}}{2}{d}^2 = \sum_{C \in \cC_k} |a_C|^2 \normLpRd{ \Chi_{C^\delta}}{2}{d}^2 + \sum_{\substack{C_1,C_2 \in \cC_k \\ C_1 \neq C_2}} a_{C_1} \overline{a_{C_2}}\inn{\Chi_{C_1^\delta}}{\Chi_{C_2^\delta}}.
    \end{equation}
    Using Lemma~\ref{lem: Lp Chi}, we may immediately bound the diagonal contribution by 
    \begin{equation}\label{eq:l2freqlocproof4}
        \sum_{C \in \cC_k} |a_C|^2 \normLpRd{ \Chi_{C^\delta}}{2}{d}^2 \lesssim \delta \sum_{C \in \cC_k} |a_C|^2.
    \end{equation}

    For the off-diagonal contribution, suppose $C_i = C(x_i, r_i) \in \cC_k$ for $i = 1$, $2$ satisfy $C_1 \neq C_2$. Then $|r_1 - r_2| \leq \delta/2$ and, by $\delta$-separation, $|x_1 - x_2| \geq \delta$, which implies 
    \begin{equation*}
       \dist(C_1,C_2) \sim \tang{C_1}{C_2}. 
    \end{equation*}
     Hence, in this case, Lemma~\ref{lem:weakorthogonality} implies
    \begin{equation}
    |\inn{\Chi_{C_1^\delta}}{\Chi_{C_2^\delta}}| \lesssim \delta \cdot (1+\delta^{-1} \dist(C_1,C_2))^{-(d+1)}.
        \label{eq:l2freqlocproof1}
    \end{equation}
    This observation motivates the following decomposition. Given $C_1 \in \cC_k$, define
    \begin{align*}
        \cC_{k,\ell}(C_1) \coloneq \big\{ C_2 \in \cC_k : 2^\ell \delta \leq \dist (C_1, C_2) < 2^{\ell + 1} \delta \big\} \qquad \textrm{for $\ell \in \N_0$.}
    \end{align*}
    Since $\cC$ is $\delta$-separated, these sets form a partition of $\cC_k \setminus \{C_1\}$.  It then follows from \eqref{eq:l2freqlocproof1} that
    \begin{align}
        \sum_{\substack{C_1,C_2 \in \cC_k \\ C_1 \neq C_2}} a_{C_1} \overline{a_{C_2}} \inn{\Chi_{C_1^\delta}}{\Chi_{C_2^\delta}}
        &\lesssim \delta\sum_{\ell = 0}^{\infty} \sum_{\substack{C_1 \in \cC_k \\ C_2 \in \cC_{k,\ell} (C_1)}} |a_{C_1}||a_{C_2}| (1+\delta^{-1} \dist(C_1,C_2))^{-(d+1)} 
        \nonumber \\
        &\lesssim  \delta\sum_{\ell = 0}^{\infty} 2^{-(d+1)\ell} \sum_{C_1 \in \cC_k}|a_{C_1}| \sum_{C_2 \in \cC_{k,\ell} (C_1)} |a_{C_2}| .
        \label{eq:l2freqlocproof5}
    \end{align}
    Again using the fact that $\cC$ is $\delta$-separated,  $ \#\cC_{k,\ell}(C_1) \lesssim 2^{\ell d}$ for all $C_1 \in \cC_k$ and $\ell \in \N_0$. Thus, by two applications of the Cauchy--Schwarz inequality,
    \begin{equation}\label{eq:l2freqlocproof6}
        \sum_{C_1 \in \cC_k}|a_{C_1}| \sum_{C_2 \in \cC_{k,\ell} (C_1)} |a_{C_2}| \lesssim 2^{\ell d} \sum_{C \in \cC_k}|a_C|^2.
    \end{equation}
    Combining \eqref{eq:l2freqlocproof5} and \eqref{eq:l2freqlocproof6}, and summing the resulting geometric series
    \begin{equation}\label{eq:l2freqlocproof7}
        \sum_{\substack{C_1,C_2 \in \cC_k \\ C_1 \neq C_2}} a_{C_1} \overline{a_{C_2}} \inn{\Chi_{C_1^\delta}}{\Chi_{C_2^\delta}}
        \lesssim \delta \sum_{C \in \cC_k}|a_C|^2.
    \end{equation}

    By combining \eqref{eq:l2freqlocproof3}, \eqref{eq:l2freqlocproof4} and \eqref{eq:l2freqlocproof7}, we obtain 
    \begin{equation*}
         \NormLpRd{\sum_{C \in \cC_k} a_C \Chi_{C^\delta}}{2}{d} \lesssim \delta^{1/2} \Brac{\sum_{C \in \cC_k}|a_C|^2}^{1/2} \qquad \textrm{for all $1 \leq k \leq K$.}
    \end{equation*}
   Taking the $\ell^2$ sum of the above estimate in $k$ and recalling \eqref{eq:l2freqlocproof2}, we conclude that
   \begin{equation*}
       \NormLpRd{\sum_{C \in \cC} a_C \Chi_{C^\delta}}{2}{d}
        \lesssim \Brac{\sum_{k=1}^K \sum_{C \in \cC_k}|a_C|^2 \, }^{1/2} \lesssim \Brac{\sum_{C \in \cC}|a_C|^2}^{1/2},
   \end{equation*}
   which is precisely the desired bound \eqref{eq:steinlpdualp2}.
\end{proof}




\subsection{Weak orthogonality via stationary phase}\label{subsec: w orthog} It remains to prove the weak orthogonality property from Lemma~\ref{lem:weakorthogonality}, which is based on oscillatory integral estimates. For $C = C(x,r)$ where $(x,r) \in \R^d \times (0,\infty)$ and $0 < \delta \leq 1/2$ dyadic, by taking the Fourier transform, 
    \begin{equation}\label{eq: Chi ft}
        (\Chi_{C^\delta}) \ft (\xi) = \delta \cdot e^{-2\pi i x \cdot \xi} \, \beta(r \delta \xi) \, \widehat{\sigma}(r\xi).
    \end{equation}
    Note $e^{-2\pi i x \cdot \xi} \, \widehat{\sigma}(r\xi)$ corresponds to the Fourier transform of the normalised measure on $C(x,r)$. The proof of Lemma~\ref{lem:weakorthogonality} relies on the stationary phase formula 
\begin{equation}\label{eq: stationary phase}
    e^{-2\pi i x \cdot \xi} \, \widehat{\sigma}(r\xi) = \sum_{ \nu \in \{-1, +1\}} e^{-2\pi i(x \cdot \xi - \nu r |\xi|)} a^{\nu}(r \xi) 
\end{equation}
where $a^{\nu} \in C^{\infty}(\widehat{\R}^d)$ are smooth symbols satisfying
\begin{equation}\label{eq: symbol bounds}
    |\partial_{\xi}^{\alpha} a^{\nu}(\xi)| \lesssim_{\alpha} \big(1 + |\xi|\big)^{-(d-1)/2 - |\alpha|} \quad \textrm{for $\alpha \in \N_0^d$ and $|\alpha| \coloneq \alpha_1 + \cdots + \alpha_d$.}
\end{equation}
see, for instance, \cite[Chapter VIII]{SteinBook} or \cite[Chapter 1]{SoggeBook}. 

\begin{proof}[Proof (of Lemma~\ref{lem:weakorthogonality})] Let $C_i = C(x_i, r_i)$ where $(x_i, r_i) \in \bbB^d \times [1,2]$ for $i = 1$, $2$. It suffices to show
\begin{equation}\label{eq: w orthog 1}
    |\inn{\Chi_{C_1^\delta}}{\Chi_{C_2^\delta}}| \lesssim_N \delta \cdot (1 + \delta^{-1} |r_1 - r_2|\big)^{-(d-1)/2} \big(1+\delta^{-1} \tang{C_1}{C_2}\big)^{-N}. 
\end{equation}
Indeed, if $\tang{C_1}{C_2} \geq \dist(C_1, C_2)/2$, then this implies \eqref{eq: w orthog} by exploiting the rapidly decaying term in \eqref{eq: w orthog 1}. On the other hand, by the triangle inequality, we have $\dist(C_1, C_2) \leq \tang{C_1}{C_2} + |r_1 - r_2|$. From this, if $\tang{C_1}{C_2} \leq \dist(C_1, C_2)/2$, then $|r_1 - r_2| \geq   \dist(C_1, C_2)/2$. Thus, \eqref{eq: w orthog 1} also implies \eqref{eq: w orthog} in this case.

Recalling \eqref{eq: Chi ft} and \eqref{eq: stationary phase}, by Plancherel's theorem and a change of variable
    \begin{equation}\label{eq: w orthog 2}
    \inn{\Chi_{C_1^\delta}}{\Chi_{C_2^\delta}} = \delta \sum_{\bdnu \in \{-1,+1\}^2}  I_{C_1, C_2}^{\delta, \bdnu}
    \end{equation}
    where
    \begin{equation*}
          I_{C_1, C_2}^{\delta, \bdnu}  \coloneq \int_{\widehat{\R}^d} e^{2\pi i \delta^{-1}\phi_{C_1, C_2}^{\bdnu}(\xi)} a_{C_1, C_2}^{\bdnu}(\xi; \delta^{-1}) \ud \xi
    \end{equation*}
for
\begin{align*}
 \phi_{C_1, C_2}^{\bdnu}(\xi) &\coloneq (x_2 - x_1) \cdot \xi + (\nu_1 r_1 - \nu_2 r_2)|\xi|, \\
 a_{C_1, C_2}^{\bdnu}(\xi; R) &\coloneq R^{(d-1)} \, a^{\nu_1}(r_1R \xi) \,  \overline{a^{\nu_2}(r_2 R \xi)} \,  \beta(r_1 \xi) \, \beta(r_2 \xi). 
\end{align*}
Given $R \geq 1$, using the properties of $\beta$ and \eqref{eq: symbol bounds}, one may show
\begin{equation*}
    \supp a_{C_1, C_2}^{\bdnu}(\,\cdot\,;R) \subseteq \bbA^d \coloneq \big\{ \xi \in \widehat{\R}^d : 1/4 \leq |\xi| \leq 2 \big\}
\end{equation*}
and
\begin{align}\label{eq: w orthog 3a}
   |\partial_{R}^N\partial_{\xi}^{\alpha}  a_{C_1, C_2}^{\bdnu}(\xi; R)| &\lesssim_{N, \alpha} R^{-N} \qquad \textrm{for all $N \in \N_0$, $\alpha \in \N_0^d$,} \\
   \label{eq: w orthog 3b}
   |\partial_{\xi}^{\alpha} \phi_{C_1, C_2}^{\bdnu}(\xi)| &\lesssim_{\alpha} 1 \qquad \textrm{for all $\alpha \in \N_0^d \setminus \{0\}$,}
\end{align}
where the second bound holds for all $\xi \in \supp a_{C_1, C_2}^{\bdnu}(\,\cdot\,;R)$. The argument now splits into two cases depending on the choice of $\bdnu$.\medskip

\noindent \underline{Case 1: $\nu_1 \nu_2 = -1$}.  Using the hypothesis that $C_1$, $C_2$ are unit scale circles,
\begin{equation*}
    |\nabla_{\xi} \phi_{C_1, C_2}^{\bdnu}(\xi)| \geq \inf_{\omega \in S^{d-1}} |x_1 - x_2 - (r_1 + r_2)\omega| \geq r_1 + r_2 - |x_1 - x_2| \geq 1 \gtrsim \tang{C_1}{C_2}. 
\end{equation*}

\noindent \underline{Case 2: $\nu_1 \nu_2 = 1$}. Here we have
\begin{equation*}
    |\nabla_{\xi} \phi_{C_1, C_2}^{\bdnu}(\xi)| \geq \inf_{\omega \in S^{d-1}} |x_1 - x_2 - (r_1 - r_2)\omega| = \tang{C_1}{C_2}.
\end{equation*}

In either case, we may now argue as follows. If $\tang{C_1}{C_2} \geq \delta$, for any $N \in \N_0$ we may apply repeated integration-by-parts, using \eqref{eq: w orthog 3a} and \eqref{eq: w orthog 3b}, to give
\begin{equation*}
    |I_{C_1, C_2}^{\delta, \bdnu}| \lesssim_N \big( 1+ \delta^{-1}\tang{C_1}{C_2}\big)^{-N}\Big|\int_{\widehat{\R}^d} e^{2\pi i \delta^{-1}\phi_{C_1, C_2}^{\bdnu}(\xi)} a_{C_1, C_2}^{\bdnu, N}(\xi; \delta^{-1}) \ud \xi\Big|,
\end{equation*}
for some choice of $a_{C_1, C_2}^{\bdnu, N} \in C^{\infty}_c(\widehat{\R}^d \times (0, \infty))$ with 
\begin{equation*}
   \supp a_{C_1, C_2}^{\bdnu, N} (\,\cdot\,; R) \subseteq \bbA^d \quad \textrm{and} \quad |\partial_{R}^M\partial_{\xi}^{\alpha} a_{C_1, C_2}^{\bdnu, N}(\xi; R)| \lesssim_{\alpha, N, M} R^{-M}
\end{equation*}
 for all $M \in \N_0$, $\alpha \in \N_0^d$ and $R \geq 1$. The Hessian of the phase satisfies 
\begin{equation*}
    \mathrm{Hess}(\phi_{C_1, C_2}^{\bdnu})(\xi) = (\nu_1r_1 - \nu_2r_2) H(\xi), \qquad \textrm{where} \qquad \operatorname{rank}H(\xi) = d-1. 
\end{equation*}
If $\nu_1 \nu_2 = -1$, then, since $C_1$, $C_2$ are unit scale, $|\nu_1r_1 - \nu_2r_2| = r_1 + r_2 \geq |r_1 - r_2|$. On the other hand, if $\nu_1 \nu_2 = 1$, then we immediately have $|\nu_1r_1 - \nu_2r_2| = |r_1 - r_2|$. Thus, in either case, by appropriately splitting up the domain of integration and applying van der Corput's lemma in $d-1$ variables (see, for instance, \cite[Theorem 1.1.4]{Sogge1991}), we obtain 
\begin{equation}\label{eq: w orthog 5}
    |I_{C_1, C_2}^{\delta, \bdnu}| \lesssim_N (1 + \delta^{-1} |r_1 - r_2|\big)^{-(d-1)/2}  \big( 1+ \delta^{-1}\tang{C_1}{C_2}\big)^{-N}, \qquad N \in \N_0,
\end{equation}
as required. The above bound continues to hold when $\tang{C_1}{C_2} \leq \delta$, simply by applying van der Corput directly without the initial repeated integration-by-parts.\medskip

Combining \eqref{eq: w orthog 2} and \eqref{eq: w orthog 5} concludes the proof. 
\end{proof}

\begin{remark} We give an additional heuristic justification for the 
\begin{equation}\label{eq: w orthog tang}
   \delta \cdot (1+\delta^{-1}\dist(C_1, C_2))^{-(d-1)/2} 
\end{equation}
factor in \eqref{eq: w orthog}. In the tangent case $\Delta(C_1, C_2) \leq \delta$, the right-hand side of \eqref{eq: w orthog} becomes \eqref{eq: w orthog tang}. Here our estimates are trivial from an oscillatory integral perspective, in the sense that they do not rely on cancellation. In particular, if we assume the (non-rigorous) heuristic $|\Chi_{C_j^{\delta}}| = \chi_{C_j^{\delta}}$, as discussed in \S\ref{sec: osc kernels}, then we may bound
\begin{equation*}
    |\inn{\Chi_{C_1^{\delta}}}{\Chi_{C_2^{\delta}}}| \leq \inn{\chi_{C_1^{\delta}}}{\chi_{C_2^{\delta}}} = |C_1^{\delta} \cap C_2^{\delta}|.
\end{equation*}
Since $C_1$, $C_2$ are close to tangent, $C_1^{\delta} \cap C_2^{\delta}$ is contained in the $\delta$-neighbourhood of a cap on $C_1$ of radius $O((1+\delta^{-1}\dist(C_1, C_2))^{-1/2})$; see Lemma~\ref{lem:intersectingannuli} for the $d= 2$ case of this assertion. The volume of such a neighbourhood is approximately \eqref{eq: w orthog tang}. 
\end{remark}



\subsection{A local spherical maximal theorem}\label{subsec: loc spherical}

We already observed that, for $d \geq 3$, Proposition~\ref{prop:steinlpdualp2} implies the $L^2$ boundedness of the spherical maximal operator. Here, as an aside, we note that Proposition~\ref{prop:steinlpdualp2} can further be used to prove the full range of boundedness for the local spherical maximal operator.

\begin{proposition}[Local spherical maximal theorem]\label{prop: loc Stein thm} For all $d \geq 3$ and $p > \frac{d}{d-1}$, we have
\begin{equation}\label{eq: loc Stein thm}
    \| M_0 f\|_{L^p(\R^d)} \lesssim_p \|f\|_{L^p(\R^d)} \qquad \textrm{for all $f \in C_c(\R^d)$.} 
\end{equation}
\end{proposition}

\begin{remark} We only consider the local maximal operator here, since Lemma~\ref{lem: loc to glob} cannot be used to obtain global maximal estimates in the full range. This is due to the restriction $p \geq 2$ arising from the use of Littlewood--Paley theory. Nevertheless, a minor modification of the argument used to prove Proposition~\ref{prop: loc Stein thm} can be used to prove the global maximal bound in Theorem~\ref{thm: Stein}. 
\end{remark}

\begin{proof}[Proof (of Proposition~\ref{prop: loc Stein thm})] Define the exponents
\begin{equation*}
    \varepsilon_d(p) \coloneq d - 1 - \frac{d}{p} \qquad \textrm{for $1 \leq p \leq 2$;}
\end{equation*}
for $p = 2$, this agrees with the definition of $\varepsilon_d(2)$ from \eqref{eq: L2 exp}. For $j \in \N$, we claim that
\begin{equation}\label{eq: Stein proof 1}
    \| M_0^j f\|_{L^p(\bbB^d)} \lesssim 2^{-j \varepsilon_d(p)} \|f\|_{L^p(\R^d)} \qquad \textrm{for $1 \leq p \leq 2$.}
\end{equation}
Once we have \eqref{eq: Stein proof 1}, we can immediately apply Lemma~\ref{lem: loc op red} to upgrade to a global estimate, with $\bbB^d$ replaced with $\R^d$. The same global bound holds for $j = 0$ by comparison with the Hardy--Littlewood maximal operator. If $d \geq 3$, then $\varepsilon_d(p) > 0$ for all $\frac{d}{d-1} < p \leq 2$ (note that the restriction $d \geq 3$ ensures that the range of $p$ is nonempty). Thus, in this case the resulting global bounds sum to give \eqref{eq: loc Stein thm} in the restricted range $\frac{d}{d-1} < p \leq 2$. The remaining range $2 < p < \infty$ is treated via interpolation with the trivial $L^{\infty}$ maximal estimate. 

By Lemma~\ref{lem: duality}, matters are further reduced to showing that, for $1 \leq p \leq 2$, the inequality
 \begin{equation}\label{eq: Stein proof 2}
        \NormLpRd{\sum_{C \in \cC} a_C \Chi_{C^\delta}}{p'}{d} \lesssim \delta^{1-d/p + \varepsilon_d(p)} \Brac{ \sum_{C \in \cC} |a_C|^{p'}}^{1/p'}
    \end{equation}
holds for all $0 < \delta \leq 1/2$ dyadic, $\cC$ any $\delta$-separated set of unit scale spheres and $(a_C)_{C \in \cC}$ any complex sequence.

For $p =1$, note that the $\delta$-separation ensures $\#\cC \lesssim \delta^{-d}$. Using Lemma~\ref{lem: Lp Chi}, we therefore obtain
    \begin{equation*}
           \NormLpRd{\sum_{C \in \cC} a_C \Chi_{C^\delta}}{\infty}{d} \lesssim \#\cC \cdot \sup_{C\in\cC} \normLpRd{\Chi_{C^\delta}}{\infty}{d}  \sup_{C \in \cC} |a_C| \lesssim \delta^{-d} \sup_{C \in \cC} |a_C|, 
    \end{equation*}
as required. The case $p = 2$ is exactly Proposition~\ref{prop:steinlpdualp2} and so \eqref{eq: Stein proof 2} follows for all $1 \leq p \leq 2$ by interpolation. 
\end{proof}




\subsection{Comparison with geometric arguments} As in \S\ref{subsec: geom comparison}, it is interesting to compare the above Fourier analytic arguments with their purely geometric counterparts. The analogue of Proposition~\ref{prop:steinlpdualp2} for the purely geometric approach to the problem is the following $L^2$ bound. 

\begin{proposition} \label{prop: geometric L2}
    For all $d \geq 3$, the inequality
\begin{equation}\label{eq: geometric L2}
    \NormLpRd{\sum_{C \in \cC} a_C \chi_{C^\delta}}{2}{d} \lesssim \delta^{1-d/2} \Brac{ \sum_{C \in \cC} |a_C|^2 }^{1/2}
\end{equation}
 holds for all $0 < \delta \leq 1/2$, any $\delta$-separated set $\cC$ of unit scale spheres and $(a_C)_{C \in \cC}$ any complex sequence. 
\end{proposition}

Proposition~\ref{prop: geometric L2} follows from a very similar argument to that used for Proposition~\ref{prop:steinlpdualp2} and the details of the proof are therefore omitted. We remark that the following bound plays the role of Lemma~\ref{lem:weakorthogonality} in the proof of Proposition~\ref{prop: geometric L2}.  

 \begin{lemma}[Spherical intersection bound]\label{lem: sphere intersection}
        Let $d \geq 3$ and $C_1$, $C_2 \subset \R^d$ be unit scale spheres. For $0 < \delta \leq 1/2$, we have
        \begin{equation}
         \inn{\chi_{C_1^{\delta}}}{\chi_{C_2^{\delta}}} =   |C_1^\delta \cap C_2^\delta| \lesssim \delta \cdot \big(1+\delta^{-1} \dist (C_1,C_2)\big)^{-1}.
            \label{eq:sphericalintersection}
        \end{equation}
    \end{lemma}

Lemma~\ref{lem: sphere intersection} is a consequence of elementary geometric considerations and calculus and we again omit the details (see, for instance, \cite[Lemma B.1]{CDK2025}). 

A significant limitation of Proposition~\ref{prop: geometric L2}, when compared with its Fourier analytic counterpart in Proposition~\ref{prop:steinlpdualp2}, is that it is no longer possible to directly interpolate the $L^2$ estimate \eqref{eq: geometric L2} to prove maximal estimates for other values of $p$. To better understand this, observe:
\begin{itemize}
    \item For the geometric approach, the goal is to prove maximal estimates which are uniform in $\delta$, as in \eqref{eq: geom loc}. This requirement is binary and unquantifiable: either the estimate is uniform or it is not.
    \item For the Fourier analytic approach, the goal is to prove maximal estimates with an additional $2^{-j\varepsilon(p)} = \delta^{\varepsilon(p)}$ decay factor, as in \eqref{eq: loc to glob a}. This requirement is inherently quantifiable in terms of the parameter $\varepsilon(p)$. For a given $p$, any $\varepsilon(p) > 0$ suffices. However, as we saw in the previous subsection, we can trade an estimate at $L^2$ with $\varepsilon(2)$ large for, say, an estimate at some lower $L^p$ with a correspondingly  smaller (but still positive!) value of $\varepsilon(p)$.
\end{itemize}
 This is one of the key advantages of the Fourier approach: it is far more amenable to interpolation. This provides flexibility to fully exploit theory in the $L^2$, $L^1$ and $L^{\infty}$ spaces, where specific tools become available. We have already observed the power of this method in the proof of the spherical maximal bounds in \S\ref{subsec: loc spherical}. In \S\ref{sec: main argument} below, we shall see that similar, but subtler, interpolation techniques (but now in the complementary $p > 2$ range) form the bedrock of Bourgain's proof of the circular maximal function. 

Finally, it is important to note that both Proposition~\ref{prop: geometric L2} and Lemma~\ref{lem: sphere intersection} \textbf{fail} for $d = 2$. Indeed, if Proposition~\ref{prop: geometric L2} were true in the $d = 2$ case, then, using Lemma~\ref{lem: geom duality}, it would imply the (fallacious) $L^2$ boundedness of the local circular maximal operator $M_0$ on $\R^2$. On the other hand, we can see the failure of the planar case of Lemma~\ref{lem: sphere intersection} directly via the following example, which is closely related to Example~\ref{ex: curv}.

\begin{example}\label{ex: tangent pair} Consider two tangent circles $C_1$, $C_2$ satisfying $\dist(C_1, C_2) \sim 1$. In this case, it is not difficult to check that $|C_1^{\delta} \cap C_2^{\delta}| \sim \delta^{3/2}$. For $\delta > 0$ small, this is much larger than $\delta^2$, which corresponds to the left-hand side of \eqref{eq:sphericalintersection}.
\end{example}

It follows that the planar version of Lemma~\ref{lem: sphere intersection} must account for tangencies between the two circles: see Lemma~\ref{lem:intersectingannuli} below. 




\section{Circle tangencies}
\label{sec:geometriclemmas}




\subsection{$\delta$-tangency} Our earlier observations suggest (interior) circle tangencies play an important role in the behaviour of the circular maximal function. Indeed, tangencies feature in both Example~\ref{ex: curv} and Example~\ref{ex: tangent pair}, suggesting that sets of circles $\cC$ with many tangent pairs constitute a key enemy. This is consistent with the weak orthogonality inequality \eqref{eq: w orthog}, which dramatically improves for transversal $C_1$, $C_2$. 

Owing to the analytic nature of the maximal problem, we are, in fact, not interested in tangencies \textit{per se}: rather, we work under a quantified tangency hypothesis.

\begin{definition} Let $C_1$, $C_2$ be circles and $0 < \delta \leq 1/2$. We say the pair $(C_1, C_2)$ is \textit{$\delta$-tangent} if $\Delta(C_1, C_2) < \delta$. 
\end{definition}

In this section, we observe a number of basic results concerning perfectly tangent circles (that is, $C_1$, $C_2$ satisfying $\Delta(C_1, C_2) = 0$). Our task is then to find \textit{continuum} analogues of these results: that is, we formulate and prove statements concerning $\delta$-tangent circles, which are quantified in terms of the $\delta$ parameter. 




\subsection{Intersecting annuli}

We begin with a quantitative description of the intersection $C_1^{\delta} \cap C_2^{\delta}$ of a pair of $\delta$-annuli in $\R^2$. For this, it is convenient to introduce the following definition. 

\begin{definition} Let $C_i = C(x_i, r_i)$ where $(x_i, r_i) \in \R^2 \times (0, \infty)$ for $i = 1$, $2$ satisfy $x_1 \neq x_2$. Then we define
\begin{equation*}
   \omega(C_1;C_2) \coloneq x_1 - r_1 \sgn(r_1-r_2) \frac{x_1-x_2}{|x_1-x_2|}. 
\end{equation*} 
\end{definition}

If $\tang{C_1}{C_2} = 0$, then $\omega(C_1;C_2)$ is the point of tangency between $C_1$ and $C_2$. If $C_1 \cap C_2 \neq \varnothing$ and $\tang{C_1}{C_2} > 0$, then $\omega(C_1;C_2)$ is the midpoint of the arc of $C_1$ between the two points of intersection between $C_1$ and $C_2$: see Figure~\ref{fig: omega} below.

\begin{figure}
    \centering
    \begin{subfigure}[t]{0.32\textwidth}
        \centering
        \includegraphics[height=4cm]{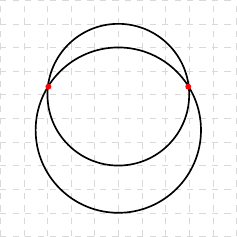}
        \captionsetup{width=0.95\textwidth, singlelinecheck=off}
        \caption{\footnotesize{Transversal pair: 
    \begin{equation*}
        \tang{C_1}{C_2} \sim \dist(C_1,C_2) \sim 1.
    \end{equation*}
        Then $C_1^{\delta} \cap C_2^{\delta}$ has diameter $O(1)$ and area $O(\delta^2)$.}}
    \end{subfigure}%
    ~ \quad
    \begin{subfigure}[t]{0.32\textwidth}
        \centering
        \includegraphics[height=4cm]{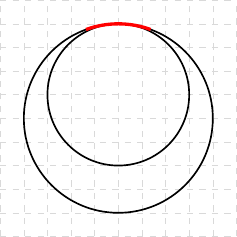}
        \captionsetup{width=0.95\textwidth, singlelinecheck=off}
        \caption{\footnotesize{Tangential pair:
        \begin{equation*}
            \tang{C_1}{C_2} \lesssim \delta;\, \dist(C_1,C_2) \sim 1. 
        \end{equation*}
        Then $C_1^{\delta} \cap C_2^{\delta}$ has diameter $O(\delta^{1/2})$ and area $O(\delta^{3/2})$.}}
    \end{subfigure}%
    ~ \quad
\begin{subfigure}[t]{0.32\textwidth}
        \centering
        \includegraphics[height=4cm]{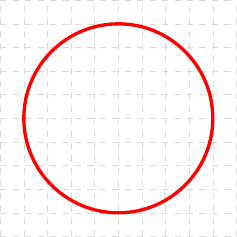}
       \captionsetup{width=0.95\textwidth, singlelinecheck=off}
        \caption{\footnotesize{Coincidence:
        \begin{equation*}
            \tang{C_1}{C_2},\, \dist(C_1,C_2) \lesssim \delta.
        \end{equation*}
        Then $C_1^{\delta} \cap C_2^{\delta}$ has diameter $O(1)$ and area $O(\delta)$.}}
    \end{subfigure}%
    \caption{Extreme cases of intersecting annuli. Note that the upper bounds for the diameter (corresponding to $\operatorname{length}(\Gamma)$) and area of $C_1^{\delta} \cap C_2^{\delta}$ are consistent with Lemma~\ref{lem:intersectingannuli} a) and \eqref{eq: int annuli area}. For (B) and (C), the lower bound for the diameter from  Lemma~\ref{lem:intersectingannuli} b) is also valid. The general statement of Lemma~\ref{lem:intersectingannuli} can be thought of as interpolating between these extremes.}
    \label{fig: intersecting annuli}
\end{figure}

\begin{lemma}[Intersecting annuli] \label{lem:intersectingannuli}
    Let $C_1$, $C_2$ be unit scale circles and $0 < \delta \leq 1/2$. 
    \begin{enumerate}[a)]
        \item The intersection
    $C_1^{\delta} \cap C_2^{\delta}$ is contained in the $\delta$-neighbourhood of an arc $\Gamma$ of $C_1$, centred at $\omega(C_1;C_2)$, with
    \begin{equation*}
        \operatorname{length}(\Gamma) \lesssim \Brac{\frac{\tang{C_1}{C_2} + \delta}{\dist (C_1,C_2) + \delta}}^{1/2}. 
    \end{equation*}
     \item If $\Delta(C_1,C_2) \leq \delta$ and $\lambda \geq 4 \delta$, then $C_1^{\delta} \cap C_2^{\lambda}$ contains the $\delta$-neighbourhood of an arc $\Gamma$ of $C_1$, centred at $\omega(C_1;C_2)$, with
\begin{equation}\label{eq: length}
    \operatorname{length}(\Gamma) \gtrsim \min\bigg\{\Brac{\frac{\lambda}{\dist(C_1,C_2) + \delta}}^{1/2}, 1 \bigg\}.
    \end{equation}

    \item There exist two (not necessarily distinct) arcs $\gamma_1$, $\gamma_2$ of $C_1$ such that $C_1^{\delta} \cap C_2^{\delta}$ is contained in the $\delta$-neighbourhood of $\gamma_1 \cup \gamma_2$ and
    \begin{equation*}
          \max_{i = 1, 2} \operatorname{length}(\gamma_i) \lesssim \frac{\delta}{(\tang{C_1}{C_2} + \delta)^{1/2}(\dist(C_1,C_2) + \delta)^{1/2}}.
        \end{equation*}
    \end{enumerate}
    
\end{lemma}

Note, in particular, that part c) implies
\begin{equation}\label{eq: int annuli area}
    |C_1^{\delta} \cap C_2^{\delta}| \lesssim \frac{\delta^2}{(\tang{C_1}{C_2} + \delta)^{1/2}(\dist(C_1,C_2) + \delta)^{1/2}},
\end{equation}
\begin{wrapfigure}{r}{0.4\textwidth}
    \centering
    \includegraphics[width=0.35\textwidth]{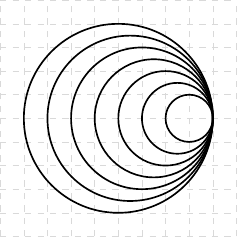}
\caption{\footnotesize{Circles arranged in a clamshell configuration. This is the unique arrangement in which all pairs $(C_1, C_2)$ are tangent.}}
    \label{fig: clam}    
\end{wrapfigure}
which is the correct formulation of Lemma~\ref{lem: sphere intersection} in the plane. Various forms of the upper bounds in Lemma~\ref{lem:intersectingannuli} a) and c) can be found, for instance, in \cite[Lemma 3]{Bourgain1986}, \cite[Lemma 4.2]{Schlag1997} and the expository article \cite[Lemma 3.1]{Wolff1999}.  We avoid reproducing the proofs here, which are moderately involved exercises in calculus and trigonometry. The lower bound in Lemma~\ref{lem:intersectingannuli} b) follows by a minor modification of the argument to prove Lemma~\ref{lem:intersectingannuli} a); see also \cite[Lemma 3.8]{PYZ}. Since it has appeared less frequently in the literature, we discuss the proof of Lemma~\ref{lem:intersectingannuli} b) in \S\ref{sec: length proof} below. In Figure~\ref{fig: intersecting annuli}, we motivate the form of the estimates by considering extremal cases.

We note the following simple consequence of Lemma~\ref{lem:intersectingannuli} c), which applies to pairs of annuli of different widths. This will be useful later in \S\ref{sec: main argument}. 

\begin{corollary}\label{cor:intersectingannuli} Let $C_1$, $C_2$ be unit scale circles and $0 < \delta \leq \lambda \leq 1/2$. Then 
\begin{equation*}
    |C_1^{\delta} \cap C_2^{\lambda}| \lesssim \delta \lambda^{1/2} (\dist(C_1, C_2) + \delta)^{-1/2}.
\end{equation*}
\end{corollary}

\begin{proof} Since $C_1^\delta \cap C_2^\lambda \subseteq C_1^\lambda \cap C_2^\lambda$, by Lemma \ref{lem:intersectingannuli} c), the intersection $C_1^\delta \cap C_2^\lambda$ is contained in the $\delta$-neighbourhood of $\gamma_1 \cup \gamma_2$, where $\gamma_1$, $\gamma_2$ are two (not necessarily distinct) arcs of $C_1$ which satisfy
\begin{align*}
    \operatorname{length}(\gamma_i) \lesssim \frac{\lambda}{(\tang{C_1}{C_2} + \lambda)^{1/2}(\dist(C_1,C_2) + \lambda)^{1/2}} \leq \frac{\lambda^{1/2}}{(\dist(C_1,C_2) + \delta)^{1/2}} 
\end{align*}
for $i = 1$, $2$. This directly implies the desired bound. 
\end{proof}




\subsection{Uniqueness of the clam shell configuration}

Given a set of circles $\cC$, we are interested in understanding the (interior) tangent pairs. The number of tangent pairs is trivially bounded by
   \begin{equation}
    \# \{ (C_1,C_2) \in \cC \times \cC : \tang{C_1}{C_2} = 0 \} \leq [\# \cC]^2.
    \label{eq:tangpairsmax}
\end{equation} 
In general, we cannot hope to do better than this upper bound. Indeed, equality holds in \eqref{eq:tangpairsmax} for $\cC$ a \textit{clamshell configuration}, where all the circles are tangent at a common point: see Figure~\ref{fig: clam}. A moment's thought shows that clamshell configuration is also the \textit{unique} arrangement that attains the upper bound in \eqref{eq:tangpairsmax}. Indeed, this is a consequence of the following elementary lemma.

\begin{lemma}[Uniqueness of clamshell configuration: discrete]\label{lem: unique clam disc} If $C_i \subset \R^2$ for $i = 0$, $1$, $2$ are circles with $\Delta(C_i, C_j) = 0$ for $0 \leq i < j \leq 2$, then $C_0  \cap C_1 \cap C_2  \neq \varnothing$.
\end{lemma}

Lemma~\ref{lem: unique clam disc} is the obvious fact that if three circles are pairwise interior tangent, then they must form a clamshell. 

For the maximal problem, rather than exact tangency, we are interested in $\delta$-tangency for $\delta > 0$. We wish to prove a continuum analogue of Lemma~\ref{lem: unique clam disc}, which applies in this quantified setting. For this, we require a preliminary definition.

\begin{definition} Let $C_i = C(x_i,r_i)$ where $(x_i, r_i) \in \R^2 \times (0, \infty)$ for $i = 0$, $1$, $2$ satisfy $x \notin \{x_1,x_2 \}$ and $r \notin \{r_1,r_2 \}$. Then we define
\[
\angle(C_1;C_0,C_2) \coloneq \angle (\sgn (r_1-r_0)(x_1-x_0), \sgn(r_1-r_2)(x_1-x_2)).
\]
\end{definition}

This is precisely the central angle of the arc of $C_1$ between the two points $\omega(C_1;C_0)$ and $\omega(C_1;C_2)$: see Figure~\ref{fig: angle}. Using these definitions, the following result is a slight reworking of \cite[Lemma 4]{Bourgain1986}. 

\begin{lemma}[Uniqueness of clamshell configuration: continuum] \label{lem:clamshell} Let $0 < \delta \leq 1/2$, $0 < \kappa \leq 1$. Suppose $C_i$ for $i = 0$, $1$, $2$ are unit scale circles such that
          \begin{equation}
 \tang{C_0}{C_i} \leq \delta, \quad i = 1,\ 2, \qquad \textrm{and} \qquad \tang{C_1}{C_2} + \delta \leq \kappa \cdot \dist(C_1,C_2). \label{eq:clamshellhyp1}
    \end{equation}
    Letting $C_\mathrm{max}, C_\mathrm{min} \in \{ C_1,C_2 \}$ denote distinct circles of maximal and minimal radius, respectively, the following hold:
    \begin{enumerate}[a)]
        \item If $r_1,r_2\geq r_0$, then $\angle(C_\mathrm{max};C_0,C_\mathrm{min}) \lesssim \kappa^{1/2}$;
        \item  If $r_1,r_2\leq r_0$, then $ \angle(C_\mathrm{min};C_0,C_\mathrm{max}) \lesssim \kappa^{1/2}$.
    \end{enumerate}
\end{lemma}

\begin{figure}
    \centering
    \begin{subfigure}[t]{0.48\textwidth}
        \centering
        \includegraphics[height=5.5cm]{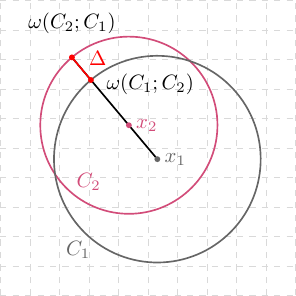}
        \captionsetup{width=0.95\textwidth, singlelinecheck=off}
        \caption{\footnotesize{Points $\omega(C_1;C_2)$ and $\omega(C_2;C_1)$ of $\delta$-tangency. The length of the line segment between these points is $\Delta = \Delta(C_1, C_2)$.}}
        \label{fig: omega}
    \end{subfigure}%
    ~ \quad
    \begin{subfigure}[t]{0.48\textwidth}
        \centering
        \includegraphics[height=5.5cm]{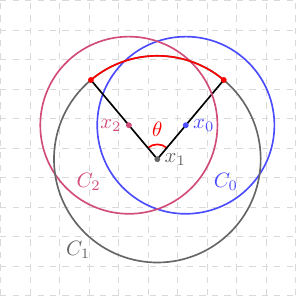}
        \captionsetup{width=0.95\textwidth, singlelinecheck=off}
        \caption{\footnotesize{The angle $\theta = \angle(C_1;C_0,C_2)$ is the central angle of the arc of $C_1$ between $\omega(C_1; C_0)$ and $\omega(C_1; C_2)$.}}
        \label{fig: angle}
    \end{subfigure}%
    \caption{Definitions featured in Lemmas~\ref{lem:intersectingannuli} and~\ref{lem:clamshell}.}
\end{figure}

Lemma~\ref{lem:clamshell} is our continuum analogue of Lemma~\ref{lem: unique clam disc}:
\begin{itemize}
    \item The hypothesis \eqref{eq:clamshellhyp1} of Lemma~\ref{lem:clamshell} is a quantified version of the pairwise tangent hypothesis in Lemma~\ref{lem: unique clam disc}.
    \item The conclusion of Lemma~\ref{lem:clamshell} tells us that the points of tangency must lie on some common arc of either $C_{\mathrm{max}}$ or $C_{\mathrm{min}}$ of bounded length. This is a quantified version of the conclusion of Lemma~\ref{lem: unique clam disc} that all three circles have a common point of tangency.
\end{itemize}

Unsurprisingly, the proof of Lemma~\ref{lem:clamshell} is based on elementary geometry. We exploit the following quantified form of the triangle inequality.

\begin{lemma}[Quantitative triangle inequality] \label{lem:quantitativetriineq}
    For $u$, $v$, $w \in \R^2$ with $u \notin \{v, w\}$, we have
    \begin{equation}
        |w-v| + |v-u| - |w-u| \geq |v-u|(1- \cos\theta) \quad \textrm{where} \quad \theta \coloneq \angle(v-u,w-u). \label{eq:quantitative1}
    \end{equation}
    Furthermore, if $|w-v| \geq |v -u|$, then
    \begin{equation*}
        |w-v| + |v-u| - |w-u| \sim |v-u|(1- \cos\theta).
    \end{equation*}
\end{lemma}

\begin{proof}
    By the law of cosines,
    \begin{equation*}
        |w-v|^2 = |v-u|^2 + |w-u|^2 -2|v-u||w-u|\cos \theta .
    \end{equation*}
    Completing the square, we then have
    \begin{equation*}
        |w-v|^2 = (|w-u| - |v - u|)^2 + 2|v-u||w-u|(1-\cos \theta ),
    \end{equation*}
    which rearranges and factorises to give
    \begin{equation}\label{eq: triangle 1}
        2|v-u||w-u|(1-\cos \theta) = (|w-v|-|w-u|+|v-u|)(|w-v|+|w-u|-|v-u|).
    \end{equation}
   
   On the one hand, by the triangle inequality, 
    \begin{equation}\label{eq: triangle 2}
        |w-v|+|w-u|-|v-u|\leq 2|w-u|.
    \end{equation}
    On the other hand, under the hypothesis $|w-v| \geq |v -u|$, we have
    \begin{equation}\label{eq: triangle 3}
        |w-v|+|w-u|-|v-u| \geq |w-u|
    \end{equation}
   Combining either \eqref{eq: triangle 2} or \eqref{eq: triangle 3} with \eqref{eq: triangle 1}, and removing common $|w-u|$ factors, concludes the proof.
\end{proof}

\begin{proof}[Proof (of Lemma~\ref{lem:clamshell})] Without loss of generality we may assume $r_1 = \max \{ r_1, r_2 \}$ and $r_2 = \min \{ r_1, r_2 \}$. We treat cases a) and b) of the lemma in parallel.\medskip

\noindent \underline{Case a): $r_1$, $r_2 \geq r_0$}. Let $\theta \coloneq \angle(C_1;C_0,C_2) =\angle (x_2-x_1, x_0-x_1)$. By the quantitative triangle inequality \eqref{eq:quantitative1} with $u \coloneq x_1$, $v \coloneq x_2$, $w \coloneq x_0$, we have
    \begin{equation}\label{eq: clam 1}
        |x_1-x_2|\theta^2 \lesssim |x_2-x_0|+|x_2-x_1|-|x_1-x_0|.
    \end{equation}
    By hypothesis, we have $||x_i-x_0|-(r_i-r_0)| \leq \delta$ for $i = 1$, $2$, and so 
    \begin{equation}\label{eq: clam 2}
        |x_2-x_0| \leq r_2-r_0 + \delta, \qquad \textrm{and} \qquad
        |x_1-x_0| \geq r_1-r_0 - \delta.
    \end{equation}

\noindent \underline{Case b): $r_1$, $r_2 \leq r_0$}. Let $\theta \coloneq \angle(C_2;C_0,C_1) =\angle (x_1-x_2, x_0-x_2)$. By the quantitative triangle inequality \eqref{eq:quantitative1} with $u \coloneq x_2$, $v \coloneq x_1$, $w \coloneq x_0$, we have
    \begin{equation}\label{eq: clam 3}
        |x_1-x_2|\theta^2 \lesssim |x_1-x_0|+|x_1-x_2|-|x_2-x_0|.
    \end{equation}
    By hypothesis, we have $||x_i-x_0|-(r_0-r_i)| \leq \delta$ for $i = 1$, $2$, and so 
    \begin{equation}\label{eq: clam 4}
         |x_1-x_0| \leq r_0-r_1 + \delta, \qquad \textrm{and} \qquad |x_2-x_0| \geq r_0-r_2 - \delta.
    \end{equation}
    
Combining either \eqref{eq: clam 1} and \eqref{eq: clam 2} in Case a) or \eqref{eq: clam 3} and \eqref{eq: clam 4} in Case b),
    \begin{equation*}
        |x_1-x_2|\theta^2 \lesssim |x_2-x_1| - (r_1-r_2) + 2\delta \leq \tang{C_1}{C_2} + 2 \delta.
    \end{equation*}
    Consequently,
    \[
    \theta^2 \lesssim \frac{\tang{C_1}{C_2} + \delta}{\dist(C_1,C_2)}.
    \]
    Finally, applying the hypothesis \eqref{eq:clamshellhyp1}, we deduce that $\theta^2 \lesssim \kappa$, as required.
\end{proof}




\subsection{Properties of the clamshell configuration} We now apply the diameter bounds for $C_1^{\delta} \cap C_2^{\delta}$ from Lemma~\ref{lem:intersectingannuli} in the special setting of an approximate clamshell configuration. Combining this with Lemma~\ref{lem:clamshell}, we are led to the following key geometric observation.

\begin{proposition} \label{prop:localisedclam}
    Let $0 < \delta$, $\rho$, $\kappa \leq 1$ with $\delta \leq \min\{1/2, \rho \kappa\}$. Suppose $C_i \subset \R^2$ for $i = 0$, $1$, $2$ are unit scale circles such that
    \begin{enumerate}[i)]
        \item $\tang{C_0}{C_i} \leq \delta$, $i = 1$, $2$, and $\tang{C_1}{C_2} + \delta \leq \kappa \cdot \dist(C_1,C_2)$;
    \item $\rho/2 \leq \dist(C_0,C_1)$, $\dist(C_0,C_2) \leq \rho$;
    \item Either $r_1,r_2\geq r_0$ or $r_1,r_2\leq r_0$.
    \end{enumerate}
    Then there is an absolute constant $K_{\circ} \geq 1$ such that
    \begin{equation*}
        C_1^{\delta} \cap C_2^{\delta} \subseteq C_0^\lambda \qquad \textrm{where} \qquad \lambda \coloneq \lambda(\rho,\kappa) \coloneq K_{\circ} \rho \kappa.
    \end{equation*}
\end{proposition}

To better understand the statement of this proposition, let $(x_i, r_i) \in \R^2 \times (0, \infty)$ for $i  = 0$, $1$, $2$ and $0 < \delta \leq \rho \leq 1$ and suppose $\tang{C_0}{C_j} \leq \delta$ and $\dist(C_0, C_j) \leq \rho$ for $j = 1$, $2$. It follows from the triangle inequality that $|r_j - r_0| \leq |x_j - x_0| + \delta$ and, consequently,
\begin{equation*}
||y - x_0| - r_0| \leq ||y - x_j| - r_j| + |r_j - r_0| + |x_j - x_0| \leq ||y - x_j| - r_j| + 3 \rho.      
\end{equation*}
We therefore have $C_j^{\delta} \subseteq C_0^{\,4\rho}$ for $j = 1$, $2$. Thus, \textit{a fortiori}, $C_1^{\delta} \cap C_2^{\delta} \subseteq C_0^{\,4\rho}$.

The above proposition shows that, under the stated hypotheses, the intersection $C_1^{\delta} \cap C_2^{\delta}$ lies in a thinner annulus, of width $O(\kappa \rho)$. For $\kappa$ small, this is a gain over the trivial containment described above. 

\begin{proof}[Proof (of Proposition~\ref{prop:localisedclam})] We shall assume $K_{\circ} \kappa \leq 4$, since otherwise the result trivially holds by the above discussion.

Hypotheses i) and iii) match those of Lemma~\ref{lem:clamshell} and therefore $C_0$, $C_1$, $C_2$ are in an approximate clamshell configuration. That is, for some $(i, j) \in \{(1,2), (2,1)\}$, we have $\angle(C_i;C_0,C_j) \lesssim \kappa^{1/2}$. Without loss of generality, by relabelling, 
    \begin{equation}
      \angle(C_1;C_0,C_2) \lesssim \kappa^{1/2}.
        \label{eq:localisedclam 1}
    \end{equation}

 On the one hand, provided $K_{\circ} \geq 4$, the hypothesis $\delta \leq \rho \kappa$ implies that $\lambda \geq 4 \delta$. By hypothesis i) we also have $\Delta(C_0,C_1) \leq \delta$. Thus, by Lemma~\ref{lem:intersectingannuli} b), the intersection $C_1^{\delta} \cap C_0^{\lambda}$ contains the $\delta$-neighbourhood $\Gamma^{\delta}$ of an arc $\Gamma$ of $C_1$, centred at $\omega(C_1;C_0)$, satisfying 
    \begin{equation}
        \operatorname{length}(\Gamma) \gtrsim \min\bigg\{\Brac{\frac{\lambda}{\dist(C_0,C_1)+ \delta} }^{1/2}, 1 \bigg\} \gtrsim K_{\circ}^{1/2}\kappa^{1/2}.
        \label{eq:localisedclam 2}
    \end{equation}

    On the other hand, by Lemma~\ref{lem:intersectingannuli} a) and hypothesis i), the intersection $C_1^{\delta} \cap C_2^{\delta}$ is contained in the $\delta$-neighbourhood $\gamma^{\delta}$ of an arc $\gamma$ of $C_1$, centred at $\omega(C_1;C_2)$, with
    \begin{equation}
        \operatorname{length}(\gamma) \lesssim \Brac{\frac{\tang{C_1}{C_2} + \delta}{\dist (C_1,C_2) + \delta}}^{1/2} \lesssim \kappa^{1/2}. \label{eq:localisedclam 3}
    \end{equation}
    
    Note that the angle between the centres $\omega(C_1;C_2)$ and $\omega(C_1;C_0)$ is $\angle(C_1;C_0,C_2)$, which is $O(\kappa^{1/2})$ by \eqref{eq:localisedclam 1}. Comparing the lengths in \eqref{eq:localisedclam 3} and \eqref{eq:localisedclam 2}, it follows that
    \begin{equation*}
        C_1^{\delta} \cap C_2^{\delta} \subseteq \gamma^{\delta} \subseteq \Gamma^{\delta} \subseteq C_0^{\lambda},
    \end{equation*}
    provided $K_{\circ} \geq 1$ is sufficiently large. This concludes the proof. 
\end{proof}




\subsection{Lower bounds for the arc length}\label{sec: length proof} Here we discuss the geometric arguments used to lower bound the length of the intersection arc in Lemma~\ref{lem:intersectingannuli} b). The proof is based on the quantitative triangle inequality from Lemma~\ref{lem:quantitativetriineq}. 

\begin{proof}[Proof (of Lemma~\ref{lem:intersectingannuli} b))] We shall only consider the case $r_1 \geq r_2$; the complementary case can be treated using similar arguments.

Let $(x_i, r_i) \in \bbB^2 \times [1, 2]$ be such that $C_i = C(x_i,r_i)$ for $i = 1$, $2$. If $y \in C_1^{\delta}$, then note that
    \begin{equation}\label{eq: lower 1 1}
        |y - x_2| \geq  r_2 + |y - x_1| - r_1 - |x_1 - x_2| + (r_1 - r_2) > r_2 - \lambda/2. 
    \end{equation}
    
    Suppose $\# \big[C_1 \cap C_2^{\lambda/2, +}\big] \leq 1$ where $C_2^{\lambda/2, +} \coloneq C(x_2, r_2 + \lambda/2)$. Since 
    \begin{equation*}
        \omega(C_1;C_2) \in C_1 \cap B(x_2, r_2 + \lambda/2),
    \end{equation*}
    it follows by continuity that $C_1 \subseteq \overline{B(x_2, r_2 + \lambda/2)}$. Thus, $C_1^{\delta} \subseteq B(x_2, r_2 + \lambda)$ and, combining this with \eqref{eq: lower 1 1}, we conclude that $C_1^{\delta} \subseteq C_2^{\lambda}$. In particular, in this case we may take $\Gamma = C_1$ with $\operatorname{length}(\Gamma) \sim 1$.\smallskip

    In light of the above, we may assume $C_1 \cap C_2^{\lambda/2, +}$ contains two points which must form the boundary of an open arc $\Gamma$ on $C_1$, centred at $\omega(C_1; C_2)$ with $\Gamma \subset C_2^{\lambda/2}$. By the triangle inequality, $C_2^{\lambda}$ contains the $\delta$-neighbourhood of $\Gamma$. Thus, matters are reduced to showing \eqref{eq: length}. 
    
Let $y \in C_1 \cap C_2^{\lambda/2, +}$, so that $|y - x_1| = r_1$ and $|y - x_2| = r_2 + \lambda/2$. We measure the angle $\theta \coloneq \angle(x_2-x_1,y-x_1)$. Since $C_1$, $C_2$ are unit scale,
    \begin{equation*}
        |x_1 - x_2| < 1 \leq r_2 < |y - x_2|.
    \end{equation*}
    Thus, by Lemma~\ref{lem:quantitativetriineq} with $u \coloneq x_1$, $v \coloneq x_2$, $w \coloneq y$, we get
    \begin{equation*}
        \theta^2 |x_1 - x_2|
        \sim
        |y - x_2| + |x_1 - x_2| - |y - x_1| 
        \geq \lambda/2 - \Delta(C_1,C_2) 
        \geq \lambda/4.
    \end{equation*}
Since $\operatorname{length}(\Gamma) \sim \theta$, this implies \eqref{eq: length}, concluding the proof in this case.
\end{proof}



\section{The main argument}\label{sec: main argument}

In this section, we combine the Fourier analytic results from \S\ref{sec: Stein} and the geometric results from \S\ref{sec:geometriclemmas} to conclude the proof of Theorem~\ref{thm: Bourgain}.




\subsection{High level strategy} Following the reductions from \S\ref{sec: Fourier prelim} and writing $q = p'$, the goal is  to prove the following dual bound.

\begin{proposition} \label{prop:bourgaindual}
    For all $1<q<2$ there exists some exponent $\varepsilon(q') > 0$ such that
    \begin{equation}
        \NormLpRd{\sum_{C \in \cC} \Chi_{C^\delta} }{q}{2} \lesssim_q \delta^{2/q-1+\varepsilon(q')} [\# \cC]^{1/q} 
        \label{eq:bourgainreduced}
    \end{equation}
    holds for $0<\delta \leq 1/2$ dyadic and $\cC$ any $\delta$-separated set of unit scale circles.
\end{proposition}

Fixing $\cC$ and $0 < \delta \leq 1/2$ as above, we can think of $(a_C)_{C \in \cC} \mapsto \sum_{C \in \cC} a_C \Chi_{C^{\delta}}$ as a linear operator mapping complex sequences indexed by $\cC$ to $L^q$ functions. The inequality \eqref{eq:bourgainreduced} then implies a restricted strong-type inequality for this operator, in the sense that it gives a bound for the restricted class of binary sequences $(a_C)_{C \in \cC}$ with $a_C \in \{0,1\}$. Assuming Proposition~\ref{prop:bourgaindual} holds, we can use restricted weak-type interpolation (see \cite[Theorem 1.4.19]{Grafakos_book}) to upgrade  \eqref{eq:bourgainreduced} to a strong-type bound for the same (open) range of exponents. In particular, for all $1<q<2$, there exists some $\varepsilon(q') > 0$ such that 
    \begin{equation*}
        \NormLpRd{\sum_{C \in \cC} a_C \Chi_{C^\delta} }{q}{2} \lesssim_q \delta^{2/q-1+\varepsilon(q')} \Brac{\sum_{C \in\cC} |a_C|^q}^{1/q} 
    \end{equation*}
    for all $0<\delta \leq 1/2$ dyadic, $\cC$ any $\delta$-separated set of unit scale circles and $(a_C)_{C \in \cC}$ any complex sequence. We may combine this bound with Lemma~\ref{lem: duality}, Lemma~\ref{lem: loc op red} and Lemma~\ref{lem: loc to glob} to conclude the proof of Theorem~\ref{thm: Bourgain}. 

We begin by comparing Proposition~\ref{prop:bourgaindual} with existing bounds coming from our discussion of Stein's spherical maximal theorem in \S\ref{sec: Stein}. By Proposition~\ref{prop:steinlpdualp2}, we have
\begin{equation}
    \NormLpRd{\sum_{C \in \cC} \Chi_{C^\delta} }{2}{2} \lesssim  [\# \cC]^{1/2}.
        \label{eq:L2bound}
\end{equation}
On the other hand, by the $p = 1$ case of Lemma~\ref{lem: Lp Chi} we have
\begin{equation}
    \NormLpRd{\sum_{C \in \cC} \Chi_{C^\delta} }{1}{2} \leq \sum_{C \in \cC} \normLpRd{ \Chi_{C^\delta} }{1}{2} \lesssim  \delta \# \cC .
        \label{eq:L1bound}
\end{equation}
Using log-convexity of $L^q$ to interpolate between \eqref{eq:L2bound} and \eqref{eq:L1bound}, we deduce that
\begin{align*}
    \NormLpRd{\sum_{C \in \cC} \Chi_{C^\delta} }{q}{2} &\lesssim \delta^{2/q-1} [\# \cC]^{1/q} \quad \text{for } 1 \leq q \leq 2.
\end{align*}
This is \textit{almost} \eqref{eq:bourgainreduced}: the only difference is the additional $\varepsilon(q')$-power of $\delta$ in \eqref{eq:bourgainreduced}. Thus, the main issue is to establish an $\varepsilon$-improvement over either of the above estimates.

It is \textbf{not} possible to directly improve the $\delta$ powers on the right-hand side of either \eqref{eq:L2bound} or \eqref{eq:L1bound} for general circle families $\cC$.\footnote{For instance, establishing \eqref{eq:L2bound} with an additional $\delta^{\varepsilon}$ factor on the right-hand side would imply the (fallacious) $L^2$-boundedness of the circular maximal function.} Instead, the strategy is as follows. We partition $\cC = \cC_\mathrm{tang} \cup \cC_\mathrm{trans}$ where, roughly speaking, 
\begin{itemize}
    \item $\cC_\mathrm{trans}$ contains circles which form very few tangent pairs;
    \item $\cC_\mathrm{tang}$ contains the rest of the circles in $\cC$.
\end{itemize}
The precise definition of this partition is presented in the next subsection. We then show in \S\ref{sec:L2trans} and \S\ref{sec:L1tang} that, for some fixed $\varepsilon_{\circ} > 0$, we have
\begin{align}
\label{eq:L2boundeps}
    \NormLpRd{\sum_{C \in \cC_\mathrm{trans}} \Chi_{C^\delta} }{2}{2} &\lesssim  \delta^{\varepsilon_{\circ}}[\# \cC]^{1/2}, \\ 
         \label{eq:L1boundeps}       
    \NormLpRd{\sum_{C \in \cC_\mathrm{tang}} \Chi_{C^\delta} }{1}{2} & \lesssim  \delta^{1+\varepsilon_{\circ}} \# \cC .
\end{align}
We stress that, in both the above estimates, it is crucial to work with the stated subfamily $\cC_{\mathrm{trans}}$ or $\cC_{\mathrm{trans}}$: in general, neither \eqref{eq:L2boundeps} nor \eqref{eq:L1boundeps} holds with $\cC$ on the left-hand side. 

Once we have \eqref{eq:L2boundeps} and \eqref{eq:L1boundeps}, given $1 < q < 2$ we may apply the triangle inequality and logarithmic convexity to bound
\begin{multline*}
    \NormLpRd{\sum_{C \in \cC} \Chi_{C^\delta} }{q}{2} \leq \NormLpRd{\sum_{C \in \cC_\mathrm{tang}} \Chi_{C^\delta}}{2}{2}^{2-2/q} \NormLpRd{\sum_{C \in \cC_\mathrm{tang}} \Chi_{C^\delta}}{1}{2}^{2/q - 1}
    \\ +
    \NormLpRd{\sum_{C \in \cC_\mathrm{trans}} \Chi_{C^\delta}}{2}{2}^{2 - 2/q}\NormLpRd{\sum_{C \in \cC_\mathrm{trans}} \Chi_{C^\delta}}{1}{2}^{2/q-1}.
\end{multline*}
We combine the estimates \eqref{eq:L2bound} and \eqref{eq:L1boundeps} for the contribution from $\cC_\mathrm{tang}$, and we combine the estimates \eqref{eq:L2boundeps} and \eqref{eq:L1bound} for the contribution from $\cC_\mathrm{trans}$. Altogether, this gives a bound of
\begin{equation*}
    \NormLpRd{\sum_{C \in \cC} \Chi_{C^\delta} }{q}{2} \lesssim \delta^{2/q - 1 + \varepsilon(q')} [\# \cC]^{1/q}, \quad \textrm{for} \quad \varepsilon(q') \coloneq \varepsilon_{\circ} \cdot \min\bigg\{ \frac{2}{q} - 1, 2 - \frac{2}{q} \bigg\} > 0.
\end{equation*}
This establishes Proposition~\ref{prop:bourgaindual} and thereby completes the proof.




\subsection{Technical preliminaries} Throughout the rest of the article, we work with a pair of fixed, small exponents $0 < \eta_{\circ} \ll \varepsilon_{\circ} < 1$. In particular, it suffices to take
\begin{equation}\label{eq: eta eps exp}
    \eta_{\circ} \coloneq 10^{-4} \qquad \textrm{and} \qquad \varepsilon_{\circ} \coloneq 10^{-2}.
\end{equation}

In order to exploit the essential support property from Lemma~\ref{lem:essentialsupp}, given a circle $C = C(x,r)$ with $(x, r) \in \R^2 \times (0, \infty)$ and $0 < \delta \leq 1/2$, we define
\begin{equation*}
    C^{\delta,*} = C^{\delta, *}(x,r) \coloneq \big\{ y\in \R^2 : ||x-y|-r| < \delta^{1-\eta_{\circ}} \big\}.
\end{equation*}
Thus, $C^{\delta,*} = C^{\delta^{1-\eta_{\circ}}}$ is a slight enlargement of $C^\delta$.




\subsection{Defining the partition}
\label{sec:partition}
Following \cite[\S5]{Bourgain1986}, we construct $\cC_{\text{trans}}$ using a greedy algorithm, successively removing circles from $\cC$ that form many tangent pairs. In particular, for $\eta_{\circ}$ and $\varepsilon_{\circ}$ as in \eqref{eq: eta eps exp}, we recursively define:
\begin{itemize}
    \item $\cC^{(0)} \coloneq \cC$.
    \item Suppose $\cC^{(0)}, \cC^{(1)}, \dots, \cC^{(J)}$ have been constructed for some $J \in \N_0$.

    \begin{itemize}
        \item If for all $C \in \cC^{(J)}$ the set 
    \[
    \cC^{(J)}_C \coloneq \{ \overline{C} \in \cC^{(J)}: \tang{C}{\overline{C}} \leq \delta^{1 - \eta_{\circ}} \}
    \]
    satisfies $\# \cC^{(J)}_C \leq \delta^{2/3 + 4 \varepsilon_{\circ}} \delta^{-2}$, then the algorithm terminates.

    \item Otherwise, there exists some $C_J \in \cC^{(J)}$ such that
    \[
    \# \cC^{(J)}_{C_J} > \delta^{2/3 +4 \varepsilon_{\circ}}\delta^{-2}
    \]
    and we define $\cC^{(J+1)} \coloneq \cC^{(J)} \setminus \cC^{(J)}_{C_J}$.
    \end{itemize}
\end{itemize}

Clearly, this algorithm terminates after finitely many steps. Let $J \in \N_0$ denote the terminal index, and define 
\begin{equation*}
   \cC_{\text{trans}} \coloneq \cC^{(J)} \qquad \textrm{and} \qquad \cC_{\text{tang}} \coloneq \cC\setminus \cC^{(J)}.
\end{equation*}
In the following subsections, we show \eqref{eq:L2boundeps} and \eqref{eq:L1boundeps} hold for these subfamilies. 




\subsection{Improved $L^2$ bound for transverse circles}
\label{sec:L2trans}

Here we verify the $L^2$ bound \eqref{eq:L2boundeps} for $\cC_\mathrm{trans}$. By the construction described in \S\ref{sec:partition}, we have
\[
\max_{C \in \cC_{\text{trans}}} \# \{ \overline{C} \in \cC_{\text{trans}}: \tang{C}{\overline{C}} \leq \delta^{1 - \eta_{\circ}} \} \leq \delta^{2/3 + 4 \varepsilon_{\circ}} \delta^{-2}.
\]
The desired estimate \eqref{eq:L2boundeps} is therefore a consequence of the following lemma.

\begin{lemma}[Transverse case $\implies L^2$ improvement] \label{lem:transversecase}
    Let $0 < \delta \leq 1/2$ be dyadic and suppose $\cA$ is a $\delta$-separated set of unit scale circles satisfying
    \begin{equation}
    \max_{C \in \cA} \# \cA_C \leq \delta^{2/3 + 4\varepsilon_{\circ}}\delta^{-2} \quad \textrm{where} \quad \cA_C \coloneq \{ \overline{C} \in \cA : \tang{C}{\overline{C} } \leq \delta^{1-\eta_{\circ}} \}.
    \label{eq:transversecasehypothesis}
    \end{equation}
    Then
    \begin{equation}
        \NormLpRd{\sum_{C \in \cA} \Chi_{C^\delta} }{2}{2} \lesssim  \delta^{\varepsilon_{\circ}}[\# \cA]^{1/2}. \label{eq:transversecasebound}
    \end{equation}
    Here $0 < \eta_{\circ} < \varepsilon_{\circ} < 1$ are as in \eqref{eq: eta eps exp}. 
\end{lemma}

\begin{proof}[Proof (of Lemma~\ref{lem:transversecase})]
Write
    \begin{equation}
        \NormLpRd{\sum_{C \in \cA} \Chi_{C^\delta} }{2}{2}^2 = \sum_{C_1,C_2 \in \cA} \inn{\Chi_{C_1^\delta}}{\Chi_{C_2^\delta}}
        = \Sigma_\mathrm{tang} + \Sigma_{\text{trans}} , \label{eq:transversecasetangtransdecomposition}
    \end{equation}
    where 
    \[
    \Sigma_\mathrm{tang} \coloneq  \sum_{C_1\in\cA} \sum_{C_2\in\cA_{C_1}} \inn{\Chi_{C_1^\delta}}{\Chi_{C_2^\delta}} \quad \text{and} \quad \Sigma_{\text{trans}} \coloneq \sum_{C_1\in\cA} \sum_{C_2\in\cA\setminus \cA_{C_1}} \inn{\Chi_{C_1^\delta}}{\Chi_{C_2^\delta}}.
    \]
    
    We first deal with the transverse contribution $\Sigma_{\text{trans}}$. If $\tang{C_1}{C_2} > \delta^{1-\eta_{\circ}}$, then the weak orthogonality inequality from Lemma~\ref{lem:weakorthogonality} implies
    \begin{equation*}
    |\inn{\Chi_{C_1^\delta}}{\Chi_{C_2^\delta}}| \lesssim_N \delta(1+\delta^{-1} \tang{C_1}{C_2})^{\ceil{N/\eta_{\circ}}} 
    \lesssim \delta^N \quad \text{for all } N \in \N.
    \end{equation*}
    By choosing $N$ sufficiently large, and using $\delta$-separation to bound $\#\cA \lesssim\delta^{-2}$, we have
    \begin{equation}
        \Sigma_\mathrm{trans} \lesssim \delta^{100} \#\cA. \label{eq:transversecasetranscontribution}
    \end{equation}
    
    We now turn to the tangential contribution $\Sigma_\mathrm{tang}$. Again by the weak orthogonality inequality,
    \[
    |\inn{\Chi_{C_1^\delta}}{\Chi_{C_2^\delta}}| \lesssim \delta\big(1+\delta^{-1} \dist(C_1,C_2)\big)^{-1/2}.
    \]
    Thus, by H\"older's inequality,
    \begin{equation}\label{lem:transversecase 1}
        \Sigma_\mathrm{tang} \lesssim \delta \cdot \sum_{C_1\in\cA} [\# \cA_{C_1}]^{3/4} \Brac{\sum_{C_2\in\cA_{C_1}} \big(1+\delta^{-1} \dist(C_1,C_2)\big)^{-2}}^{1/4}.
    \end{equation}
   By a dyadic decomposition argument, the centre separation ensures that
\begin{equation}\label{lem:transversecase 2}
        \sum_{C_2\in\cA} \big(1+\delta^{-1} \dist(C_1,C_2)\big)^{-2} \lesssim \log \delta^{-1}.
    \end{equation}
    Combining \eqref{lem:transversecase 1} and \eqref{lem:transversecase 2} with the hypothesis \eqref{eq:transversecasehypothesis}, we deduce that
    \begin{equation}
        \Sigma_\mathrm{tang} \lesssim \delta \cdot \sum_{C_1\in\cA} \delta^{-(1-3\varepsilon_{\circ})} (\log\delta^{-1})^{1/4} 
        \lesssim \delta^{2\varepsilon_{\circ}} \#\cA . \label{eq:transversecasetangcontribution}
    \end{equation}
    Plugging the bounds \eqref{eq:transversecasetranscontribution} and \eqref{eq:transversecasetangcontribution} into the right-hand side of \eqref{eq:transversecasetangtransdecomposition} yields the desired bound \eqref{eq:transversecasebound}.
\end{proof}

\subsection{Improved $L^1$ bound for tangent circles}
\label{sec:L1tang}

It remains to verify that $\cC_\mathrm{tang}$ satisfies \eqref{eq:L1boundeps}. This is the crux of the proof. 

By the construction described in \S\ref{sec:partition}, there exist circles $C_0, \dots, C_{J-1} \in \cC$ such that $\cC_{\text{tang}}$ may be decomposed as a disjoint union
\begin{align*}
    \cC_{\text{tang}} = \bigcup_{j = 0}^{J-1} \cC^{(j)}_{C_j} \qquad \textrm{where} \qquad \# \cC^{(j)}_{C_j} \geq \delta^{2/3 + 4\varepsilon_{\circ}}\delta^{-2} \qquad \textrm{for $0 \leq j \leq J-1$.}
\end{align*}
In order to prove \eqref{eq:L1boundeps}, it then suffices to show
\begin{equation}\label{eq: tang key}
   \NormLpRd{\sum_{C \in \cC^{(j)}_{C_j}} \Chi_{C^\delta}}{1}{2} \lesssim \delta^{1+\varepsilon_{\circ}} \# \cC^{(j)}_{C_j} \qquad \textrm{for all $0 \leq j \leq J-1$.} 
\end{equation}
The desired bound \eqref{eq:L1boundeps} then follows by summing, using the disjointness of the $\cC^{(j)}_{C_j}$.

The key estimate \eqref{eq: tang key} is a consequence of the following lemma.

\begin{lemma}[Tangent case $\implies L^1$ improvement]
\label{lem:L1tang}
    Let $0 < \delta \leq 1/2$ be dyadic and suppose $\cA$ is a $\delta$-separated family of unit scale circles satisfying:
    \begin{enumerate}[i)]
        \item There exists some $C_0 \in \cA$ such that $\tang{C}{C_0} \leq \delta^{1-\eta_{\circ}}$ for all $C \in \cA$;
        \item $\# \cA \geq \delta^{2/3 + 4\varepsilon_{\circ}} \delta^{-2}$.
    \end{enumerate}   
    Then
    \begin{align*}
        \NormLpRd{\sum_{C \in \cA} \Chi_{C^\delta}}{1}{2} \lesssim \delta^{1+\varepsilon_{\circ}} \# \cA.
    \end{align*}
    Here $0 < \eta_{\circ} < \varepsilon_{\circ} < 1$ are as in \eqref{eq: eta eps exp}. 
\end{lemma}

The proof of Lemma~\ref{lem:L1tang} is the most involved part of the argument and relies heavily on the geometric lemmas discussed in \S\ref{sec:geometriclemmas}.

\begin{proof}[Proof (of Lemma~\ref{lem:L1tang})]
    We present the argument in steps.\smallskip




    \noindent\textit{Step 1: Dyadic pigeonholing distances}. Write $\delta = 2^{-j}$ for $j \in \N$ and define 
    \begin{align*}
        \cA_0 &\coloneq \{ C\in\cA:\dist(C,C_0) < 2^{-\floor{j/2}} \}, \\
        \cA_\ell &\coloneq \{ C\in\cA:2^{\ell-1-\floor{j/2}} \leq \dist(C,C_0) < 2^{\ell-\floor{j/2}}\} \qquad \textrm{for $1 \leq \ell \leq \floor{j/2}$,}
    \end{align*}
    so we have the partition $\cA = \bigcup_{\ell = 0}^{L(\delta)} \cA_\ell$ for $L(\delta) \coloneq \floor{j/2}$.
    By pigeonholing, it suffices to show
    \begin{equation}
    \NormLpRd{\sum_{C \in \cA_\ell} \Chi_{C^\delta}}{1}{2} \lesssim \delta^{1+2\varepsilon_{\circ}} \# \cA \quad \text{for $0 \leq \ell \leq L(\delta)$. } \label{eq:dyadicsufficestoshow}
    \end{equation}
    By additional pigeonholing, we may also assume either $r \geq r_0$ for all $C(x,r) \in \cA_\ell$ or $r \leq r_0$ for all $C(x,r) \in \cA_\ell$.

    The contribution from $\cA_0$ may be estimated as in \eqref{eq:L1bound}, giving
    \begin{align*}
        \NormLpRd{\sum_{C \in \cA_0} \Chi_{C^\delta}}{1}{2} \leq \sum_{C \in \cA_0} \normLpRd{ \Chi_{C^\delta}}{1}{2} \lesssim \delta \# \cA_0.
    \end{align*}
    Since the circles in $\cA_0$ have $\delta$-separated centres, we have $\#\cA_0 \lesssim \delta^{-1}$. Furthermore, by hypothesis i) of the lemma,
    \begin{equation*}
        \delta^{-1} = \delta^{1/3 - 4 \varepsilon_{\circ}} \delta^{2/3 + 4 \varepsilon_{\circ}} \delta^{-2} \leq \delta^{1/3 - 4 \varepsilon_{\circ}} \#\cA.
    \end{equation*}
    Combining these observations, 
    \begin{equation*}
        \NormLpRd{\sum_{C \in \cA_0} \Chi_{C^\delta}}{1}{2} \lesssim \delta^{1 + 1/3 - 4 \varepsilon_{\circ}} \#\cA \leq \delta^{1 + 2 \varepsilon_{\circ}} \#\cA
    \end{equation*}
    for $\varepsilon_{\circ} \coloneq 10^{-2}$ as in \eqref{eq: eta eps exp}. \smallskip




    \noindent\textit{Step 2: Spatial partition}. Henceforth, we let $1 \leq \ell \leq L(\delta)$ and define 
    \begin{equation}\label{eq: parameters}
        \rho \coloneq 2^{\ell-\floor{j/2}} \sim 2^{\ell} \delta^{1/2}, \qquad  \kappa \coloneq \delta^{5\varepsilon_{\circ}}, \qquad \lambda = \lambda(\rho,\kappa) \coloneq K_{\circ} \rho\kappa,
    \end{equation}
    where the latter is as in Proposition~\ref{prop:localisedclam}.

    Let $\psi_{C_0^\lambda} \in C_c^\infty (\R^2)$ satisfy $0 \leq \psi_{C_0^\lambda}(x) \leq 1$ for all $x \in \R^2$ and
    \begin{equation*}
        \psi_{C_0^\lambda}(x) = 1 \quad \textrm{if $x \in C_0^\lambda$,} \qquad \psi_{C_0^\lambda}(x) = 0 \quad \textrm{if $x \notin C_0^{2\lambda}$,} \qquad \|| \nabla \psi_{C_0^\lambda}|\|_{L^{\infty}(\R^2)} \lesssim \lambda^{-1}.
    \end{equation*}
    Proposition~\ref{prop:localisedclam} implies that for $C_1$, $C_2 \in \cA_\ell$ sufficiently close to tangent, $C_1^\delta \cap C_2^\delta$ is contained in the annulus $C_0^\lambda$. This motivates the decomposition
    \begin{align}
        \NormLpRd{\sum_{C \in \cA_\ell} \Chi_{C^\delta}}{1}{2} \leq \NormLpRd{\sum_{C \in \cA_\ell} \Chi_{C^\delta} \cdot \psi_{C_0^\lambda}}{1}{2} + \NormLpRd{\sum_{C \in \cA_\ell} \Chi_{C^\delta} (1 -\psi_{C_0^\lambda} )}{1}{2}. \label{eq:localpsidecomposition}
    \end{align}




    \noindent\textit{Step 3: The local contribution}. For the first term on the right-hand side of \eqref{eq:localpsidecomposition}, we exploit the localisation to bound the $L^1$ norm directly. The key observation is that for each $C \in \cA_{\ell}$, the annulus $C^{\delta}$ only intersects $C_0^{\lambda}$ along a (fairly) short arc.
    
    By the essential support property from Lemma~\ref{lem:essentialsupp}, observe that
    \begin{equation*}
        \NormLpRd{\sum_{C \in \cA_\ell} \Chi_{C^\delta} \cdot \psi_{C_0^\lambda}}{1}{2} \leq \sum_{C \in \cA_\ell} \norm{ \Chi_{C^\delta} }_{L^1(C_0^{2\lambda})} \lesssim \sum_{C \in \cA_\ell} |C^{\delta,*} \cap C_0^{2\lambda}| + \delta^{100}.
    \end{equation*}
    By Corollary~\ref{cor:intersectingannuli}, we may bound
    \begin{equation*}
        |C^{\delta,*} \cap C_0^{2\lambda}| 
        \lesssim \delta^{1-\eta_{\circ}} \lambda^{1/2}\dist(C,C_0)^{-1/2}
        \lesssim \delta^{1-\eta_{\circ}} \kappa^{1/2}.
    \end{equation*}
    Recalling the definition of $\kappa$ from \eqref{eq: parameters}, this implies that
    \begin{equation}
        \NormLpRd{\sum_{C \in \cA_\ell} \Chi_{C^\delta} \cdot \psi_{C_0^\lambda}}{1}{2} \lesssim \delta^{1-\eta_{\circ}} \delta^{5\varepsilon_{\circ}/2} \# \cA 
        \lesssim \delta^{1+2\varepsilon_{\circ}} \# \cA,  \label{eq:localpsibound}
    \end{equation}
    which provides a favourable bound for this term.\smallskip




    \noindent\textit{Step 4: Non-local contribution: bilinearisation and trichotomy}. We now consider the second term on the right-hand side of \eqref{eq:localpsidecomposition}. Following the observations below the statement of Proposition~\ref{prop:localisedclam}, we have
    \begin{equation*}
        C^{\delta,*} \subseteq C_0^{\,4\rho} \qquad \textrm{for all $C \in \cA_\ell$.}
    \end{equation*}
    Thus, using the essential support property and the Cauchy--Schwarz inequality,
    \begin{equation}
        \NormLpRd{\sum_{C \in \cA_\ell} \Chi_{C^\delta} (1 -\psi_{C_0^\lambda} )}{1}{2} \lesssim \rho^{1/2} \NormLpRd{\sum_{C \in \cA_\ell} \Chi_{C^\delta} (1 -\psi_{C_0^\lambda} )}{2}{2} + \delta^{100}. \label{eq:nonlocall1tol2}
    \end{equation}

    By passing to the $L^2$ norm in \eqref{eq:nonlocall1tol2}, we have effectively `bilinearised' the problem. In particular, we express the $L^2$ norm in terms of a weighted inner product
    \begin{equation}
        \Norm{\sum_{C \in \cA_\ell} \Chi_{C^\delta} (1 -\psi_{C_0^\lambda} )}_{L^2(\R^2)}^2 
        = \sum_{C_1,C_2 \in \cA_\ell} \inn{\Chi_{C_1^\delta}}{\Chi_{C_2^\delta}}_{L^2(u)} \label{eq:innerprodnonlocal}
    \end{equation}
    where
    \begin{align*}
    u \coloneq |(1-\psi_{C_0^\lambda})|^2 \quad \textrm{and} \quad    \inn{f}{g}_{L^2(u)} \coloneq \int_{\R^2} f(x) \overline{g(x)} u(x) \ud x \quad \textrm{for $f$, $g \in L^2(u)$.}
    \end{align*}
    We remark that $0 \leq u \leq 1$ and $\supp u \subseteq \R^2 \setminus C_0^{\lambda}$. Thus, we are interested in studying interactions between pairs $\Chi_{C_1^\delta}$, $\Chi_{C_2^\delta}$, as measured by $\inn{\Chi_{C_1^\delta}}{\Chi_{C_2^\delta}}_{L^2(u)}$.

The remainder of the argument is roughly as follows. We split the right-hand side of \eqref{eq:innerprodnonlocal} into contributions from tangent and transversal pairs.
\begin{itemize}
    \item Suppose $C_1$, $C_2 \in \cA_{\ell}$ are tangent. Proposition~\ref{prop:localisedclam} tells us $C_1^{\delta} \cap C_2^{\delta} \subseteq C_0^{\lambda}$. On the other hand, the cutoff $u$ is supported \textbf{away} from $C_0^{\lambda}$. In particular, $\inn{\Chi_{C_1^\delta}}{\Chi_{C_2^\delta}}_{L^2(u)}$ is negligible in this case, since the functions $\Chi_{C_1^\delta}\overline{\Chi_{C_2^\delta}}$ and $u$ have essentially disjoint support.
    \item Suppose $C_1$, $C_2 \in \cA_{\ell}$ are transversal. As in Lemma~\ref{lem:weakorthogonality}, the oscillatory integral $\inn{\Chi_{C_1^\delta}}{\Chi_{C_2^\delta}}_{L^2(u)}$ enjoys additional cancellation.\footnote{Unfortunately, we cannot apply Lemma~\ref{lem:weakorthogonality} directly here, due to the presence of the cutoff $u$. Nevertheless, an appropriate variant of Lemma~\ref{lem:weakorthogonality} can be applied.} This results in a favourable estimate for the sum over all transversal pairs, by arguing in the same spirit as the proof of Lemma~\ref{lem:transversecase}. 
\end{itemize}

To make the above scheme precise, we perform a trichotomy of the indexing set $\cA_\ell \times \cA_\ell$. For each $C_1 \in \cA_\ell$, we write $\cA_\ell$ as a disjoint union
    \begin{align*}
        \cA_\ell = \cA_\mathrm{diag}(C_1) \cup \cA_\mathrm{tang}(C_1) \cup \cA_\mathrm{trans}(C_1)
    \end{align*}
    where
    \begin{align*}
        \cA_\mathrm{diag}(C_1) &\coloneq \{ C_2 \in \cA_\ell: \dist(C_1,C_2) \leq \tfrac{2}{\kappa} \cdot \delta^{1-\eta_{\circ}}\}, \\
        \cA_\mathrm{tang}(C_1) &\coloneq \{ C_2 \in \cA_\ell: \dist(C_1,C_2) \geq \tfrac{2}{\kappa} \cdot \delta^{1-\eta_{\circ}};\, \tang{C_1}{C_2} \leq \tfrac{\kappa}{2} \cdot \dist(C_1,C_2)\}, \\
        \cA_\mathrm{trans}(C_1) &\coloneq \{ C_2 \in \cA_\ell: \dist(C_1,C_2) \geq \tfrac{2}{\kappa} \cdot \delta^{1-\eta_{\circ}};\, \tang{C_1}{C_2} > \tfrac{\kappa}{2} \cdot \dist(C_1,C_2)\}.
    \end{align*}
    This induces a corresponding decomposition of the right-hand side sum in \eqref{eq:innerprodnonlocal}. We treat each term individually.\smallskip




    \noindent\textit{Step 5: Diagonal term}. By the pointwise bound $0 \leq u \leq 1$ and three applications of the Cauchy--Schwarz inequality,
    \begin{align*}
    \Sigma_{\mathrm{diag}} &\coloneq \sum_{C_1 \in \cA_\ell} \sum_{C_2 \in \cA_\mathrm{diag}(C_1) } |\inn{\Chi_{C_1^\delta}}{\Chi_{C_2^\delta}}_{L^2(u)}| \\
    &\leq \sum_{C_1 \in \cA_\ell} \normLpRd{\Chi_{C_1^\delta}}{2}{2} \sum_{C_2 \in \cA_\mathrm{diag}(C_1) } \normLpRd{\Chi_{C_2^\delta}}{2}{2} \\
    &\leq \max_{C \in \cA_\ell} \# \cA_\mathrm{diag}(C) \sum_{C \in \cA_\ell} \normLpRd{\Chi_{C^\delta}}{2}{2}^2.
    \end{align*}
    Since $\cA_\ell$ is $\delta$-separated, using the definitions in \eqref{eq: parameters}, we have 
    \begin{align*}
        \max_{C \in \cA_\ell} \# \cA_\mathrm{diag}(C) \lesssim (\kappa^{-1}\delta^{1-\eta_{\circ}}/\delta)^{2} \lesssim\delta^{-11\varepsilon_{\circ}}.
    \end{align*}
    Thus, by the $p = 2$ case of Lemma~\ref{lem: Lp Chi}, we get
    \begin{equation}\label{eq: step 5 1}
        \Sigma_{\mathrm{diag}} \lesssim \delta^{-11\varepsilon_{\circ}} \sum_{C \in \cA_\ell} \normLpRd{\Chi_{C^\delta}}{2}{2}^2 
        \lesssim \delta^{1-11\varepsilon_{\circ}} \# \cA.
    \end{equation}
    Recall from hypothesis ii) that $\# \cA \geq \delta^{2/3 + 4\varepsilon_{\circ}} \delta^{-2}$, which implies that
    \begin{equation}
    \# \cA = \frac{1}{\# \cA} [\# \cA]^2 \leq \delta^{4/3 - 4 \varepsilon_{\circ}} [\# \cA]^2. \label{eq:sAsquarebound}
    \end{equation}
    Combining \eqref{eq: step 5 1} and \eqref{eq:sAsquarebound}, we deduce that
    \begin{equation}
        \Sigma_{\text{diag}} \lesssim \delta^{2 + 4\varepsilon_{\circ}} \delta^{1/3 - 19 \varepsilon_{\circ}} [\# \cA]^2
        \lesssim\delta^{2 + 4 \varepsilon_{\circ}} [\# \cA]^2,\label{eq:trichotomynonlocaldiagbound}
    \end{equation}
which provides a favourable estimate for this term. \smallskip



    \noindent\textit{Step 6: Tangent term}. Note that $\supp u \subseteq \R^2 \setminus C_0^\lambda$ and $0 \leq u \leq 1$, so that 
    \begin{align}
        \Sigma_{\text{tang}} &\coloneq \sum_{C_1 \in \cA_\ell} \sum_{C_2 \in \cA_\mathrm{tang}(C_1) } |\inn{\Chi_{C_1^\delta}}{\Chi_{C_2^\delta}}_{L^2(u)}| \nonumber \\
        &\lesssim \sum_{C_1 \in \cA_\ell} \sum_{C_2 \in \cA_\mathrm{tang}(C_1) } |C_1^{\delta,*} \cap C_2^{\delta,*} \cap \R^2 \setminus C_0^\lambda| + \delta^{100} [\# \cA]^2, \label{eq:tangcasegeombound}
    \end{align}
    where we have again used the essential support property from Lemma~\ref{lem:essentialsupp}. Observe that for $C_1 \in \cA_\ell$, $C_2 \in \cA_\mathrm{tang} (C_1)$ we have
    \begin{enumerate}[i)]
        \item$\tang{C_0}{C_1}, \, \tang{C_0}{C_2} \leq \delta^{1-\eta_{\circ}}$ by hypothesis i) of the lemma and
    \begin{align*}
        \tang{C_1}{C_2} + \delta^{1-\eta_{\circ}} \leq \kappa \cdot \dist(C_1, C_2)
    \end{align*}
    by the definition of $\cA_\mathrm{tang}(C_1)$;
        \item $\rho/2 \leq \dist(C_0,C_1), \, \dist(C_0,C_2) \leq \rho$ by the definition of $\cA_\ell$ and \eqref{eq: parameters};

        \item $r_1,r_2  \geq r$ or $r_1,r_2 \leq r$, by our initial pigeonholing.
    \end{enumerate}
     Thus, we may apply Proposition~\ref{prop:localisedclam} to conclude that
    \begin{align*}
        C_1^{\delta,*} \cap C_2^{\delta, *} \subseteq C_0^\lambda.
    \end{align*}
    Thus, all the terms of the sum on the right-hand side of \eqref{eq:tangcasegeombound} are zero and
    \begin{equation}
        \Sigma_{\text{tang}} \lesssim \delta^{100} [\# \cA]^2, \label{eq:trichotomynonlocaltangbound}
    \end{equation}
which is certainly a favourable estimate.\smallskip




    \noindent\textit{Step 7: Transverse term}. Finally, we consider the transverse term
    \begin{equation*}
        \Sigma_{\text{trans}} \coloneq \sum_{C_1 \in \cA_\ell} \sum_{C_2 \in \cA_\mathrm{trans}(C_1) } |\inn{\Chi_{C_1^\delta}}{\Chi_{C_2^\delta}}_{L^2(u)}|. 
    \end{equation*}
    We use the transversality to obtain a good bound for the $|\inn{\Chi_{C_1^\delta}}{\Chi_{C_2^\delta}}_{L^2(u)}|$, as in the transverse case considered in Lemma~\ref{lem:transversecase}. However, we cannot appeal directly to the weak orthogonality lemma here. This is because the inner products are not over the entire space $\R^2$, but a restricted domain due to the cutoff $u$ (which, \textit{a priori}, could potentially be a bad domain that allows little to no cancellation). Nevertheless, a variant of the weak orthogonality lemma does hold. 

    \begin{claim}
        For $C_1 \in \cA_\ell$ and $C_2 \in \cA_\mathrm{trans}(C_1)$, we have
        \begin{equation}
            |\inn{\Chi_{C_1^\delta}}{\Chi_{C_2^\delta}}_{L^2(u)}| \lesssim \lambda^{-1} \frac{\delta^{1-\eta_{\circ}}}{\tang{C_1}{C_2}} |C_1^{\delta,*} \cap C_2^{\delta,*}| + \delta^{100}. \label{eq:trichotomylocaltransinnerbound}
        \end{equation}
    \end{claim}

     We postpone the proof of the claim until \S\ref{sec:concluding} below. Assuming \eqref{eq:trichotomylocaltransinnerbound} for now, we may apply this bound together with \eqref{eq: int annuli area} to deduce that
    \begin{align*}
        \Sigma_\mathrm{trans} &\lesssim \sum_{C_1 \in \cA_\ell} \sum_{C_2 \in \cA_\mathrm{trans}(C_1) } \lambda^{-1} \frac{\delta^{1-\eta_{\circ}}}{\tang{C_1}{C_2}} 
        |C_1^{\delta,*} \cap C_2^{\delta,*}| 
        + \delta^{100} [\# \cA]^2 \\
        &\lesssim \delta^{3(1-\eta_{\circ})} \lambda^{-1} \sum_{C_1 \in \cA_\ell} \sum_{C_2 \in \cA_\mathrm{trans}(C_1) } \tang{C_1}{C_2}^{-3/2} \dist(C_1,C_2)^{-1/2} + \delta^{100} [\# \cA]^2.
    \end{align*}
    By definition, $ \tang{C_1}{C_2} > (\kappa/2) \cdot \dist(C_1,C_2)$ for all $C_2 \in \cA_\mathrm{trans}(C_1)$, and so
    \begin{align}
        \Sigma_\mathrm{trans} &\lesssim \delta^{1-3\eta_{\circ}} \kappa^{-5/2} \rho^{-1} \sum_{C_1, C_2 \in \cA_\ell} (1+\delta^{-1} \dist (C_1,C_2))^{-2} + \delta^{100} [\# \cA]^2 \nonumber \\
        &\lesssim \delta^{1-3\eta_{\circ}-25\varepsilon_{\circ}/2} \rho^{-1} \log\delta^{-1} \cdot \# \cA + \delta^{100} [\# \cA]^2 \nonumber \\
        &\lesssim \delta^{2+4\varepsilon_{\circ}} \rho^{-1} [\# \cA]^2, \label{eq:transboundfinal}
    \end{align}
    again by \eqref{eq:sAsquarebound} and our choice of $\eta_{\circ} \coloneq 10^{-4}$ and $\varepsilon_{\circ} \coloneq 10^{-2}$ from \eqref{eq: eta eps exp}.\smallskip




    \noindent\textit{Step 8: Concluding the argument}. From \eqref{eq:nonlocall1tol2} we have
    \begin{equation*}
        \NormLpRd{\sum_{C \in \cA_\ell} \Chi_{C^\delta} (1 -\psi_{C_0^\lambda} )}{1}{2} \lesssim \rho^{1/2} \big(\Sigma_\mathrm{diag} + \Sigma_\mathrm{tang} + \Sigma_\mathrm{trans}\big)^{1/2} + \delta^{100}.
    \end{equation*}
    By \eqref{eq:trichotomynonlocaldiagbound}, \eqref{eq:trichotomynonlocaltangbound} and \eqref{eq:transboundfinal}, we have
    \begin{align*}
        \Sigma_\mathrm{diag} + \Sigma_\mathrm{tang} + \Sigma_\mathrm{trans} \lesssim \rho^{-1} \cdot \delta^{2+4\varepsilon_{\circ}} [\#\cA]^2
    \end{align*}
    and so
    \begin{align*}
        \NormLpRd{\sum_{C \in \cA_\ell} \Chi_{C^\delta} (1 -\psi_{C_0^\lambda} )}{1}{2} \lesssim \delta^{1+2\varepsilon_{\circ}} \#\cA.
    \end{align*}
    Plugging this, along with \eqref{eq:localpsibound}, into \eqref{eq:localpsidecomposition}, we obtain the desired bound \eqref{eq:dyadicsufficestoshow}.
\end{proof}




\subsection{A variant of the stationary phase estimate}
\label{sec:concluding}

To conclude, we address the proof of the claim from Step 7. Similar bounds are also studied in \cite[\S3]{Sogge1991}.

\begin{proof}[Proof (of Claim)]
        The argument is based on modifying the proof of the weak orthogonality inequality from Lemma~\ref{lem:weakorthogonality}. First note, by the pointwise bound $0 \leq u \leq 1$ and Lemma~\ref{lem:essentialsupp}, we have
        \begin{equation}\label{eq: step 7 clm 0}
          |\inn{\Chi_{C_1^\delta}}{\Chi_{C_2^\delta}}_{L^2(u)}| \lesssim |C_1^{\delta,*} \cap C_2^{\delta,*}| + \delta^{100}  
        \end{equation}
        and so \eqref{eq:trichotomylocaltransinnerbound} holds trivially if $\lambda^{-1} \delta^{1-\eta_{\circ}} \gtrsim \tang{C_1}{C_2}$.
        Thus, henceforth we assume
        \begin{equation}
            \tang{C_1}{C_2} \geq K \lambda^{-1} \delta^{1-\eta_{\circ}}
            \label{eq: step 7 clm 1}
        \end{equation}
        where $K \geq 1$ is a constant, chosen sufficiently large so as to satisfy the requirements of the forthcoming argument. 
        
        We may further assume that $C_1^{\delta,*} \cap C_2^{\delta,*} \neq \varnothing$, since otherwise the claim holds trivially by \eqref{eq: step 7 clm 0}. Under this hypothesis, it is easy to see that
        \begin{align*}
            |r_1-r_2| \leq |x_1 -x_2| + 2\delta^{1-\eta_{\circ}} \quad \text{ and so } \quad \tang{C_1}{C_2} \lesssim \dist(C_1,C_2) +\delta^{1-\eta_{\circ}}.
        \end{align*}
        Combining this with \eqref{eq: step 7 clm 1}, we have
        \begin{align*}
            \lambda \gtrsim K \cdot \frac{\delta^{1-\eta_{\circ}}}{\tang{C_1}{C_2} ^{1/2}\dist (C_1,C_2)^{1/2}}.
        \end{align*}
        By Lemma~\ref{lem:intersectingannuli} c), provided $K \geq 1$ is chosen sufficiently large, each connected component of $C_1^{\delta,*} \cap C_2^{\delta,*} $ is contained in a square of sidelength~$\lambda$.

        We decompose $u$ into $u_1$ and $u_2$ where each $u_j$ is supported in a square of sidelength $\lambda$ and satisfies
        \begin{equation*}
        |\partial_x^\alpha u_j (x)| \lesssim_{\alpha} \lambda^{-|\alpha|} \quad \text{for all } x \in \R^2, \ \alpha \in \N_0^2
        \end{equation*}
        and
        \begin{equation*}
        \inn{\Chi_{C_1^\delta}}{\Chi_{C_2^\delta}}_{L^2(u)} = \sum_{j=1,2} \inn{\Chi_{C_1^\delta}}{\Chi_{C_2^\delta}}_{L^2(u_j)} + \delta^{100}.
        \end{equation*}
        Using the Fourier multiplication formula, we then write
        \begin{align*}
            \inn{\Chi_{C_1^\delta}}{\Chi_{C_2^\delta}}_{L^2(u_j)} = \int_{\R^2} (\Chi_{C_1^\delta} \cdot \overline{\Chi_{C_2^\delta}})(x) u_j (x) \ud x = \int_{\widehat{\R}^2} (\Chi_{C_1^\delta}) \ft  * \big(\,\overline{\Chi_{C_2^\delta}}\,\big) \ft (\zeta) \cdot \hat{u}_j (\zeta) \ud \zeta
        \end{align*}
        Recalling the Fourier transform formula \eqref{eq: Chi ft}, we have
        \begin{multline*}
            (\Chi_{C_1^\delta}) \ft  * (\overline{\Chi_{C_2^\delta}}) \ft (\zeta) \\
            = \delta^2 \int_{\widehat{\R}^2} e^{-2 \pi i x_1 \cdot (\zeta - \xi )}\beta(r_1 \delta  (\zeta - \xi)) \, \hat{\sigma}(r_1 (\zeta - \xi)) \,e^{-2\pi i x_2 \cdot \xi}   \beta(r_2 \delta  \xi) \,\overline{  \hat{\sigma}(r_2 (-\xi)) } \,   \ud \xi.
        \end{multline*}
        Similarly to the proof of Lemma~\ref{lem:weakorthogonality}, we apply the stationary phase formula \eqref{eq: stationary phase} for $\hat{\sigma}$. In this way,
        \begin{equation*}
          \inn{\Chi_{C_1^\delta}}{\Chi_{C_2^\delta}}_{L^2(u_j)} = \delta^2 \sum_{\bdnu \in \{-1,+1\}^2} I_{C_1,C_2}^{\delta, j, \bdnu}
        \end{equation*}
        where
        \begin{equation*}
            I_{C_1,C_2}^{\delta, j, \bdnu} \coloneq \int_{\widehat{\R}^2} \int_{\widehat{\R}^2} e^{2\pi i \tilde{\phi}_{C_1, C_2}^{\bdnu}(\xi, \zeta)} \tilde{a}_{C_1, C_2}^{\bdnu, \delta}(\xi, \zeta)\ud \xi\,\hat{u}_j(\zeta)\ud \zeta
        \end{equation*}
        for
        \begin{align*}
            \tilde{\phi}_{C_1, C_2}^{\bdnu} (\xi, \zeta) &\coloneq -(x_2-x_1)\cdot  \xi + \nu_1 r_1 |\zeta - \xi| 
            - \nu_2 r_2|\xi| - x_1 \cdot \zeta, \\
            \tilde{a}_{C_1, C_2}^{\bdnu}(\xi, \zeta; R) &\coloneq a^{\nu_1}(r_1(\zeta - \xi))  \,
            \overline{a^{\nu_2}(r_2 \xi)}  \,
            \beta(r_1 R^{-1} (\zeta - \xi)) \,
            \beta(r_2R^{-1} \xi), \quad R \geq 1. 
        \end{align*}
        Here the $a^{\nu}$ for $\nu \in \{-1, +1\}$ are as in \eqref{eq: stationary phase}; in particular, they satisfy \eqref{eq: symbol bounds}. Our goal is now to show
        \begin{equation}\label{eq: step 7 clm 2}
            |I_{C_1,C_2}^{\delta, j, \bdnu}| \lesssim \delta^{100} \qquad \textrm{for $j = 1$, $2$ and $\bdnu \in \{-1, +1\}^2$.}
        \end{equation}
        
        Rescaling and decomposing the $\zeta$-integration, to prove \eqref{eq: step 7 clm 2} for fixed values of $j$ and $\bdnu$, it suffices to show
 \begin{equation}\label{eq: step 7 clm 3}
         \sum_{\ell = 1, 2} \Big|\int_{Z_{\ell}} \int_{\widehat{\R}^2} e^{2\pi i \delta^{-1} \phi_{C_1, C_2}^{\bdnu} (\xi, \zeta; \delta^{-1}, \lambda^{-1})} a_{C_1, C_2}^{\bdnu}(\xi, \zeta; \delta^{-1}, \lambda^{-1}) \ud \xi \,w_j(\zeta; \lambda^{-1})\,\ud \zeta \Big| \lesssim \delta^{100}
        \end{equation}
        where
        \begin{equation*}
            Z_1 \coloneq B(0, \delta^{-\eta_{\circ}}) \qquad \textrm{and} \qquad Z_2 \coloneq \widehat{\R}^2 \setminus B(0, \delta^{-\eta_{\circ}}) 
        \end{equation*}
        and
        \begin{gather*}
            \phi_{C_1, C_2}^{\bdnu} (\xi, \zeta; R, S) \coloneq R^{-1} \tilde{\phi}_{C_1, C_2}^{\bdnu}(R\xi, S \zeta), \;  a_{C_1, C_2}^{\bdnu} (\xi, \zeta; R, S) \coloneq R^2\tilde{a}_{C_1, C_2}^{\bdnu}(R\xi, S\zeta; R), \\
            w_j(\zeta; S) \coloneq S^2 \hat{u}_j (S \zeta), \qquad j = 1,\,2 \textrm{ and } R, S \geq 1.
        \end{gather*}
        Given $R$, $S \geq 1$ and $\zeta \in \widehat{\R}^2$, one may show 
        \begin{equation*}
            \supp a_{C_1, C_2}^{\bdnu}(\,\cdot\,, \zeta; R, S) \subseteq \bbA^2 \coloneq \{\xi \in \widehat{\R}^2 : 1/4 \leq |\xi| \leq 2 \big\} 
        \end{equation*}
        and
        \begin{align}\label{eq: step 7 clm 4a}
            |\partial_{\xi}^{\alpha} a_{C_1, C_2}^{\bdnu}(\xi, \zeta; R, S)| &\lesssim_{\alpha} 1 \qquad \textrm{for all $\alpha \in \N_0^2$;} \\
            \label{eq: step 7 clm 4b}
            |\partial_{\xi}^{\alpha} \phi_{C_1, C_2}^{\bdnu}(\xi, \zeta; R, S)| &\lesssim_{\alpha} 1 \qquad \textrm{for all $\alpha \in \N_0^2 \setminus \{0\}$;} \\
            \label{eq: step 7 clm 4c}
            | w_j(\zeta; \lambda^{-1})| &\lesssim_N ( 1 + |\zeta|)^{-N} \qquad \textrm{for all $N \in \N_0$.}
        \end{align}
            
       It follows from \eqref{eq: step 7 clm 4c} that
        \begin{equation*}
            |w_j(\zeta; \lambda^{-1})| \lesssim \delta^{100}(1 + |\zeta|)^{-10} \qquad \textrm{for all $\zeta \in Z_2$.}
        \end{equation*}
        This immediately implies a favourable bound for the $\ell = 2$ term in \eqref{eq: step 7 clm 3}. 
                
        It remains to bound the $\ell = 1$ term in \eqref{eq: step 7 clm 3}, where the $\zeta$-integration is restricted to $|\zeta| < \delta^{-\eta_{\circ}}$. Note that the $\xi$-gradient of the phase is
        \begin{equation}\label{eq: step 7 clm 5}
         (\nabla_{\xi} \phi_{C_1, C_2}^{\bdnu}) (\xi, \zeta; \delta^{-1}, \lambda^{-1}) = - (x_2-x_1) + (\nu_1 r_1-\nu_2 r_2) \frac{\xi}{|\xi|}  + \nu_1 r_1 \Brac{\frac{\xi - \lambda^{-1} \delta\zeta}{|\xi - \lambda^{-1} \delta\zeta|} - \frac{\xi}{|\xi|}}.
        \end{equation}
        By our hypothesis \eqref{eq: step 7 clm 1}, provided $K \geq 1$ is chosen sufficiently large, 
        \begin{equation}\label{eq: step 7 clm 6}
            \inf_{t \in [0,1]} |\xi - t\lambda^{-1} \delta\zeta| \geq |\xi| - \lambda^{-1} \delta |\zeta| \geq |\xi| - \lambda^{-1} \delta^{1- \eta_{\circ}} \gtrsim 1
        \end{equation}
        for all $(\xi, \zeta) \in \bbA^2 \times Z_1$. Thus, by the mean value theorem (using \eqref{eq: step 7 clm 6} to bound the relevant derivative) and another application of \eqref{eq: step 7 clm 1}, we have 
        \begin{equation}\label{eq: step 7 clm 7}
            \Abs{\frac{\xi - \lambda^{-1} \delta\zeta}{|\xi - \lambda^{-1} \delta\zeta|} - \frac{\xi}{|\xi|}} \lesssim \lambda^{-1} \delta|\zeta| < \lambda^{-1} \delta^{1 - \eta_{\circ}} \leq K^{-1} \tang{C_1}{C_2} 
        \end{equation}
        for all $(\xi, \zeta) \in \bbA^2 \times Z_1$. We now briefly consider two cases, following the dichotomy in the proof of Lemma~\ref{lem:weakorthogonality}.\medskip

        \noindent \underline{Case 1: $\nu_1 \nu_2 = -1$}. Using the hypothesis that $C_1$, $C_2$ are unit scale circles,
        \begin{equation*}
            \inf_{\omega \in S^1}|x_2 - x_1 - (\nu_1r_1-\nu_2r_2) \omega| =\inf_{\omega \in S^1}|x_2 - x_1 - (r_1 + r_2) \omega| \geq 1 \gtrsim \tang{C_1}{C_2}.
        \end{equation*}
        \noindent \underline{Case 2: $\nu_1 \nu_2 = 1$}. Here we have
        \begin{equation*}
            \inf_{\omega \in S^1}|x_2 - x_1 - (\nu_1r_1-\nu_2r_2) \omega| = \inf_{\omega \in S^1}|x_2 - x_1 - (r_1 - r_2)\omega| = \tang{C_1}{C_2}.
        \end{equation*}
            
         In either case, by combining \eqref{eq: step 7 clm 5} and \eqref{eq: step 7 clm 7} with the preceding observations,
        \begin{align}
        \nonumber
          |(\nabla_{\xi} \phi_{C_1, C_2}^{\bdnu}) (\xi, \zeta; \delta^{-1}, \lambda^{-1})| &\geq  \inf_{\omega \in S^1}|x_2 - x_1 - (r_1-r_2) \omega| - O\big(K^{-1} \tang{C_1}{C_2}\big) \\
          \label{eq: step 7 clm 8}
          & \gtrsim \tang{C_1}{C_2}
        \end{align}
 for all $(\xi, \zeta) \in \bbA^2 \times Z_1$,  provided $K \geq 1$ is chosen sufficiently large.
 
        Finally, recalling the hypothesis $C_2 \in \cA_{\mathrm{trans}}(C_1)$, it follows from \eqref{eq: step 7 clm 8} that
        \begin{equation*}
            |(\nabla_{\xi} \phi_{C_1, C_2}^{\bdnu}) (\xi, \zeta; \delta^{-1}, \lambda^{-1})| \gtrsim \kappa \cdot \dist (C_1,C_2) \gtrsim  \delta^{1-\eta_{\circ}}
        \end{equation*}
         for all $(\xi, \zeta) \in \bbA^2 \times  Z_1$.  
        In view of \eqref{eq: step 7 clm 4a} and \eqref{eq: step 7 clm 4b}, repeated integration-by-parts now implies a favourable bound for the $\ell = 1$ term in \eqref{eq: step 7 clm 3}. This completes the proof.
    \end{proof}




\bibliography{Reference}
\bibliographystyle{amsplain}

\end{document}